\documentclass[fullpage, 11pt]{article}
\usepackage{amsfonts}
\usepackage{amssymb}
\usepackage{amsmath}
\usepackage{latexsym}
\usepackage{amsthm}
\usepackage{epsfig}
\usepackage{graphicx}
\usepackage{subfig}
\usepackage{color}
\usepackage{url}
\usepackage{enumerate}
\usepackage{ifpdf}  

\graphicspath{{figures/}}

\DeclareMathOperator	\argmax		{arg\,max}

\DeclareMathOperator	\conv		{conv}
\DeclareMathOperator	\cone 		{cone}

\DeclareMathOperator	\ext			{ext}
\DeclareMathOperator	\intr			{int}

\DeclareMathOperator	\rec			{rec}

\DeclareMathOperator	\relint		{rel\,int}

\DeclareMathOperator	\verts		{vert}

\newcommand{\bb}{\mathbb}

\newcommand{\R}{\bb R}
\newcommand{\Q}{\bb Q}
\newcommand{\Z}{\bb Z}

\newcommand\st{\mid}
\newcommand\bigst{\mathrel{\big|}}
\newcommand\Bigst{\mathrel{\Big|}}

\newtheorem{theorem}{Theorem}[section]
\newtheorem{corollary}[theorem]{Corollary}
\newtheorem{lemma}[theorem]{Lemma}
\newtheorem{proposition}[theorem]{Proposition}
\newtheorem{remark}[theorem]{Remark}
\newtheorem{observation}[theorem]{Observation}

\newtheorem{definition}[theorem]{Definition}

\newcommand{\Rf}{R_f}		

\newcommand{\y}{\bar{y}}			
\newcommand{\p}{\bar{p}}			
\renewcommand{\a}{\bar{a}}		

\newcommand{\T}{\mathcal{T}}		
\renewcommand{\S}{\mathcal{S}}	
\newcommand{\Y}{\mathcal{Y}}   	

\newcommand{\M}{M}    			

\newcommand{\Blocking}[1]{#1^\vee}  

\makeatletter
\renewenvironment{proof}[1][\proofname]{\par
  \pushQED{\qed}%
  \normalfont \topsep6\p@\@plus6\p@\relax
  \trivlist
  \item[\hskip\labelsep
        \itshape\bfseries
    #1\@addpunct{.}]\ignorespaces
}{%
  \popQED\endtrivlist\@endpefalse
}
\makeatother

\begin{document}

\title{Algorithmic and Complexity Results for Cutting Planes Derived from Maximal Lattice-Free Convex Sets}

\author{
Amitabh Basu\thanks{Dept. of Mathematics, University of California, Davis, 
\texttt{abasu@math.ucdavis.edu}} \and
Robert Hildebrand\thanks{Dept. of Mathematics, University of California, Davis,
\texttt{rhildebrand@math.ucdavis.edu}} \and
Matthias K\"oppe\thanks{Dept. of Mathematics, University of California, Davis, 
\texttt{mkoeppe@math.ucdavis.edu} }
}

\date{July 25, 2011\thanks{$\relax$Revision: 201 $ - \ $Date: 2011-07-25 12:55:40 -0700 (Mon, 25 Jul 2011) $ $}}

\maketitle

\begin{abstract}
We study a mixed integer linear program
with $m$ integer variables and $k$ non-negative continuous variables
in the form of the relaxation of the corner polyhedron 
that was introduced by Andersen, Louveaux, Weismantel and
Wolsey [\emph{Inequalities from two rows of a simplex tableau}, Proc.\ IPCO
2007, LNCS, vol.~4513, Springer, pp.~1--15]. 
We
describe the facets of this mixed integer linear program via the extreme
points of a well-defined polyhedron. We then utilize this description to give
polynomial time algorithms to derive valid inequalities with optimal $l_p$
norm for arbitrary, but fixed $m$.  For the case of $m=2$, we give a
refinement and a new proof of a characterization of the
facets by Cornu\'ejols and Margot [\emph{On the facets of mixed integer programs
  with two integer variables and two constraints}, Math.\ Programming
\textbf{120} (2009), 429--456].  The key point of our approach is that the conditions are much more explicit and can be tested in a more direct manner, removing the need for a reduction
algorithm. These results allow us to show that the
relaxed corner polyhedron has only polynomially many facets.   
\end{abstract}

\section{Introduction}

The integer programming community has recently focused on developing a unifying theory for cutting planes. This has involved applying tools from convex analysis and the geometry of numbers to combine the ideas behind Gomory's corner polyhedron~\cite{gom} and Balas' intersection cuts~\cite{bal} into one uniform framework. It is fair to say that this recent line of research was started by the seminal paper by Andersen, Louveaux, Weismantel and Wolsey~\cite{alww}, which took a fresh look at the work done by Gomory and Johnson in the 1960's. We refer the reader to~\cite{corner_survey} for a survey of these results. 

It can be argued that the theoretical research has tended to emphasize the structural aspects of these cutting planes and the algorithmic aspects have not been developed as intensively. Our goal in this paper is to derive structural results which, we hope, will be useful from an algorithm design perspective. Hence, our emphasis is on deriving polynomiality results about the structure of these cutting planes. We also provide concrete polynomial time algorithms for generating the ``best'' or ``deepest'' cuts, according to some standard criteria.

To this end, we study the following system, introduced by Andersen et al.~\cite{alww} and Borozan and Cornu\'ejols~\cite{bc}.
\begin{equation}\label{M_fk}
\begin{aligned}
x & =  f + \sum_{j = 1}^k r^js_j, \\
x & \in \mathbb{Z}^m, \quad s_j \geq 0 \quad\textrm{for all }j = 1,\dots,k.
\end{aligned}
\end{equation}

We will assume that the data is rational, i.e., $f \in \Q^m$ and $r^j \in \Q^m$ for all $j \in \{1, \ldots, k\}$. This model appears as a natural relaxation of Gomory's corner polyhedron~\cite{gom}. As mentioned above, this model has received significant attention in recent years for developing the theory behind cutting planes derived from multiple rows of the optimal simplex tableaux.  Note that
to describe the solutions of~\eqref{M_fk}, one only needs to record
the values of the $s_j$ variables. We use $R_f = R_f(r^1, \ldots, r^k)$ to
denote the set of all points $s$ such that \eqref{M_fk} is
satisfied. It is well-known that all valid inequalities for $\conv(\Rf)$, where $\conv$ denotes the convex hull, can be derived using the Minkowski functional of maximal lattice-free convex sets. We state this formally in Theorem~\ref{thm:int_cuts} below. In this paper we give algorithms and theorems about the facet structure of $\conv(\Rf)$, which are expected to be useful for generating strong cutting planes for general mixed integer linear programs.

\paragraph{Motivation and Results.} It is well-known that the integer hull $\conv(\Rf)$ is a polyhedron of the {\em blocking type}. In Section~\ref{sec:alg}, we first describe the so-called {\em blocking polyhedron} for $\conv(\Rf)$. This is the convex set of all valid inequalities for $\conv(\Rf)$. For a detailed account of blocking polyhedra and the ``polar'' set of the valid inequalities for such polyhedra, see Chapter~9 in~\cite{sch}. The main result of Section~\ref{sec:alg} gives an explicit description of the blocking polyhedron of $\conv(\Rf)$ using a polynomial number of inequalities (Theorem~\ref{thm:main}). This implies that all facets of $\conv(\Rf)$ can be obtained by enumerating the extreme points of a polyhedron with a polynomial number of facets in the dual space. This result has the same flavor as Gomory's result for describing all facets of the corner polyhedron implicitly via the extreme points of a well-defined polyhedron (see Theorem~18 in~\cite{gom}).

We next exploit this to provide efficient algorithms for finding the optimal
valid inequality according to certain norms of the coefficient vector. 
More precisely, let $\Vert v \Vert_p = (\sum_{j=1}^k|v_j|^p)^{1/p}$ be the standard $l_p$ norm of a vector $v \in \R^k$. If $\sum_{j=1}^k \gamma_j s_j \geq 1$ is a valid inequality for $\conv(\Rf)$, its $l_p$ norm is $\Vert \gamma \Vert_p$ where $\gamma$ is the vector in $\R^m$ with components $\gamma_j$. We give polynomial time algorithms to determine cuts with minimum $l_p$ norm for arbitrary, but fixed $m$. For the special case of the $l_1$ and $l_\infty$ norms, this reduces to solving a linear program with polynomially many constraints. We also give an alternative approach for the $l_\infty$ norm.

We then investigate the case of $m=2$ in more detail in
Sections~\ref{sec:nec_cond} and \ref{sec:poly_facets}. In particular, we show
that the number of facets of $\conv(\Rf)$ 
is polynomial in the input. This result is proved in
Section~\ref{sec:poly_facets} (Theorem~\ref{THM:POLYFACETS}). In order to
prove this theorem, we first develop some technology in Section~\ref{sec:tilting}
to derive necessary conditions for a valid inequality to be a facet. Our hope
is that these tools can be utilized to prove useful theorems about facets of
$\conv(\Rf)$ for $m \geq 3$, in the same vein as the results
of Cornu\'ejols and Margot appearing in~\cite{cm}. Although we do not derive
such results in this paper, we exhibit the promise of this approach by giving alternative
proofs of necessary conditions for inequalities to be facets which appear
in~\cite{cm} and providing more refined and new necessary conditions. The necessary conditions in~\cite{cm} are stated as a particular termination condition of a complicated algorithm. This makes them hard to be used in a practical setting. In contrast, our refined conditions are explicit and can be tested directly. This makes them much more useful from the practical point of view of actually generating facet defining cutting planes.  Another advantage of our technique over the Cornu\'ejols\kern1pt
--\kern1pt Margot proof is that when the necessary conditions are violated, we
can explicitly express the given valid inequality as a convex combination of
other valid inequalities. This is crucial in obtaining a proof of the fact
that the so-called {\em triangle closure} is a polyhedron~\cite{bhk}. This settles an important open problem in this recent line of research. Finally, and perhaps most importantly, we envision that the ideas behind the
polynomiality results of Section~\ref{sec:poly_facets} can be exploited to
design algorithms and heuristics for deriving effective cutting planes. We
emphasize this by using the constructive nature of the proof for
Theorem~\ref{THM:POLYFACETS} to give a polynomial time algorithm for
enumerating all the facets of $\conv(\Rf)$ for $m=2$
(Theorem~\ref{thm:enumerate}). 

We mention here that some variations of these ideas have been explored by Louveaux and Poirrier~\cite{lp}, and also by Fukasawa et al.~\cite{fukasawa}.

\section{Preliminaries}\label{sec:prelim}

It is well-known that $\conv(\Rf)$ is a full-dimensional
polyhedron of blocking type, i.e., $\conv(\Rf) \subset \R^k_+$ (where $\R^k_+$
denotes the nonnegative orthant) and if $x\in \conv(\Rf)$, then $y\geq x$
implies $y \in \conv(\Rf)$. Hence, all nontrivial valid inequalities for
$\conv(\Rf)$ can be written as $\gamma \cdot s =
\sum_{j=1}^k \gamma_j s_j \geq 1$ for some vector $\gamma \in \R^k_+$
(see~\cite{sch}, Chapter 9 for more details on polyhedra of blocking type).  

A valid inequality $\smash{\sum_{j=1}^k\gamma_j s_j \geq 1}$ for $\conv(\Rf)$ is called {\em minimal} if it is not dominated by another inequality, i.e., there does not exist a {\em different} valid inequality $\smash{\sum_{j=1}^k \gamma'_j s_j \geq 1}$ such that $\gamma'_j \leq \gamma_j$ for $j=1,\dots, k$. 
A valid inequality $\gamma  \cdot s \geq 1$ for $\conv(\Rf)$ is called \emph{extreme} if there do not exist valid inequalities $\gamma^1 \cdot s \geq 1$, $\gamma^2 \cdot s \geq 1$ such that $\gamma = \frac{1}{2} \gamma^1 + \frac{1}{2} \gamma^2$. 
For polyhedra of blocking type, extreme inequalities are always minimal. Moreover, since $\conv(\Rf)$ is full-dimensional, facets and extreme inequalities for $\conv(\Rf)$ are one and the same thing. We now collect the main results from the recent theory of cutting planes using lattice-free sets. 
For more details, please see~\cite{corner_survey}.


\begin{definition} Let $K\subset \R^m$ be a closed convex set containing the
origin in its interior. The {\em gauge} or the \emph{Minkowski functional} is defined by
$$\psi_K(x)=\inf\{\,t>0\st  t^{-1}x\in K \,\} \quad \mbox{ for all }
x\in \R^m.$$ 
\end{definition}
By definition $\psi_K$ is non-negative.

\begin{theorem}[Intersection cuts \cite{bal}, \cite{corner_survey}]
\label{thm:int_cuts}
Consider any closed convex set $M$ containing the point $f$ in its
interior, but no integer point in its interior. Let $K = M - f$.
Then the inequality $\sum_{j=1}^k\psi_K(r^j)s_j \geq 1$ is valid
for $\conv(\Rf)$. Moreover, every valid inequality of $\conv(\Rf)$ can be derived in this manner.
\end{theorem}
For convenience, we also say that the function $\psi_K$ is extreme when the corresponding inequality $\sum_{j=1}^k \psi_K(r^j)s_j \geq 1$ is extreme. We will refrain from using the terminology that $\psi_K$ defines a facet of $\conv(\Rf)$ as to not confuse these facets with facets of lattice-free polytopes. 
We will work with a fixed set of rays $\{r^1, \ldots, r^k\}\subset \R^m$. The interior of any set $M \subseteq \R^m$ will be denoted by $\intr(M)$.

It is also well-known (see~\cite{corner_survey}) that all minimal inequalities (and hence all extreme inequalities) can be derived using {\em maximal lattice-free convex sets}, i.e., convex sets containing no integer point in their interior that are maximal with respect to set inclusion. Moreover, it is known~\cite{bccz, lovasz} that maximal lattice-free convex sets are polyhedra whose recession cones are not full-dimensional. Since we will be concerned with maximal lattice-free convex sets with $f$ in their interior, one can represent such sets in the following canonical manner. 

Let $B \in \R^{n\times m}$ be a matrix with $n$ rows $b^1,\dots,b^n\in\R^m$.
We write $B = (b^1; \dots; b^n)$. 
Let 
\begin{equation}\label{eq:M(B)}
\M(B) = \{\,x \in \R^m \st b^i\cdot (x-f) \leq 1\text{ for $i=1, \dots, n$}\,\}.
\end{equation}
This is a polyhedron with $f$ in its interior. We will denote its vertices by $\verts(B)$. In fact, any polyhedron with $f$ in its interior can be given such a description. We will mostly deal with matrices $B$ such that $\M(B)$ is a maximal lattice-free convex set in $\R^m$. 

This description enables one to describe the Minkowski functional by a simple piecewise-linear formula:

\begin{theorem}[see \cite{basu}, Theorem 24]
\label{thm:formula}
Let $B \in \R^{n\times m}$ such that the recession cone of $M(B)$ is not
full-dimensional (i.e., $b^i\cdot r \leq 0$ has no solution satisfying all constraints at strict inequality). Then,
\begin{equation}\label{eq:formula}\psi_{M(B)-f} (r) = \max_{i\in \{1, \ldots, n\}} b^i \cdot r.\end{equation}
\end{theorem}

Therefore, all minimal inequalities for $\conv(\Rf)$ can be derived using~\eqref{eq:formula} from matrices $B$ such that $M(B)$ is a maximal lattice-free convex set in $\R^m$. For convenience of notation, for any matrix $B\in \R^{n\times m}$ we define $\psi_B(r) = \psi_{M(B)-f} (r) = \max_{i\in \{1, \ldots, n\}} b^i \cdot r$.

For the case of $m=2$, Lov\'asz characterized the maximal lattice-free convex sets in $\R^2$ as follows.
\begin{theorem}[Lov\'asz \cite{lovasz}]
\label{mlfcb2}
In the plane, a maximal lattice-free convex set with non-empty interior is one of the following:
\begin{enumerate}
\item A split $c \leq a x_1 + b x_2 \leq c+1$ where $a$ and $b$ are co-prime integers and $c$ is an integer;
\item A triangle with an integral point in the interior of each of its edges;
\item A quadrilateral containing exactly four integral points, with exactly one of them in the interior of each of its edges.  Moreover, these four integral points are vertices of a parallelogram of area 1.
\end{enumerate}

\end{theorem}

Following Dey and Wolsey \cite{dw}, the maximal lattice-free triangles can be further partitioned into
three canonical types (see Figure 1):
\begin{itemize}
\item \emph{Type 1 triangles}: triangles with integral vertices and exactly one integral point in the
relative interior of each edge;
\item \emph{Type 2 triangles}: triangles with at least one fractional vertex $v$, exactly one integral
point in the relative interior of the two edges incident to $v$ and at least two integral
points on the third edge;
\item \emph{Type 3 triangles}: triangles with exactly three integral points on the boundary, one in
the relative interior of each edge.
\end{itemize}
Figure 1 shows these three types of triangles as well as a maximal lattice-free quadrilateral and a split satisfying the properties of Theorem \ref{mlfcb2}.


\begin{figure}
\centering
\ifpdf
\includegraphics[scale = .6]{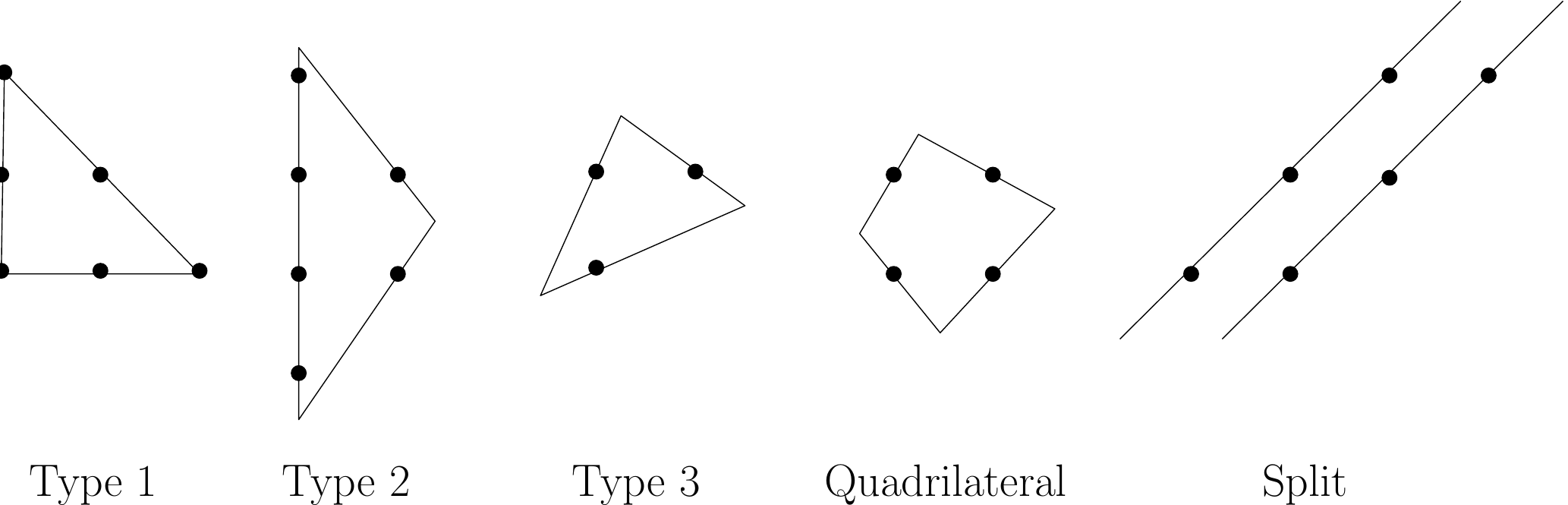}
\else
\includegraphics[scale = .6]{lattice_free_sets.eps}
\fi
%
\caption{Types of maximal lattice-free convex sets in $\R^2$}
\end{figure}

\section{Description and algorithmic results for the set of all valid
  inequalities for \boldmath$\conv(\Rf)$}\label{sec:alg} 

For the results of this section, we will assume that the conical hull of the
set of rays $\{r^1, \ldots, r^k\}$ is $\R^m$. This simplifies the arguments
presented and implies $k > m$. 

\subsection{Polyhedral structure}

As mentioned in Section~\ref{sec:prelim}, $\conv(\Rf)$ is a polyhedron of blocking type. We will study the {\em blocking polyhedron} of $\conv(\Rf)$, i.e., 
$$\Blocking{\conv(\Rf)} = \bigl\{\, \gamma \in \R^k_+ \bigst \gamma \cdot s \geq 1 \textrm{ for all } s \in \conv(\Rf)\,\bigr\}.$$ 
This is the set of all normal vectors of nontrivial valid inequalities for
$\conv(\Rf)$. We refer to \cite{sch} for a discussion of polyhedra of blocking
type and these related notions.  
It is well-known that for any polyhedron $P$ of blocking type, the set $\Blocking{P}$ is a polyhedron.  

In this section, we give an explicit description of $\Blocking{\conv(\Rf)}$. Moreover,
when $m$ is fixed (not part of the input), our description of $\Blocking{\conv(\Rf)}$ will have polynomially many inequalities. 
From the definitions, it follows that the extreme inequalities for $\conv(\Rf)$ are given by the extreme points of $\Blocking{\conv(\Rf)}$. It is well-known that for a full-dimensional polyhedron like $\conv(\Rf)$, facets and extreme inequalities are equivalent concepts.\smallbreak

We start with the following version of Carath\'eodory's theorem.

\begin{lemma}\label{lem:caratheo}
Let $P$ be a polyhedron given by $P = \conv(\{v^1, \ldots, v^p\}) + \cone(\{r^1, \ldots, r^q\})$ with $dim(P) = n$. For any $x \in P$, there exist subsets $I \subseteq \{1,\ldots, p\}$ and $J \subseteq \{1,\ldots, q\}$ such that
\begin{enumerate}[\rm(i)]
\item $|I| + |J| \leq n+1$,
\item $x \in \conv(\{\,v^i \st i \in I\,\}) + \cone(\{\,r^j\st j \in J\,\})$.
\end{enumerate}
\end{lemma}

The lemma follows immediately by the standard homogenization of $P$ and then applying Carath\'eodory's theorem for cones.\smallbreak

Let $\mathcal{I}$ be the set of all subsets~$I$ of $\{1, \ldots, k\}$ such that $\{\,r^j \st j\in I\,\}$ is a basis for $\R^m$. Given any $x \in \Z^m$ and $I \in \mathcal{I}$ such that $x - f \in \cone(\{\,r^j \st r^j \in I\,\})$, let $s_j(x, I)$ be the (non-negative) coefficient of $r^j$ when $x - f$ is expressed in the basis $\{\,r^j \st j \in I\,\}$. Moreover, for any set $I\in \mathcal{I}$, $X(I)$ is the set of all $x \in \Z^m$ such that $x - f \in \cone(\{\,r^j \st j \in I\,\})$.

\begin{proposition}\label{prop:main}

\begin{equation}\label{eq:B(Rf)}
\Blocking{\conv(\Rf)} = \Bigl\{\, \gamma \geq 0 \Bigst \sum_{j \in I} \gamma_j s_j(x,I)  \geq 1  \quad \forall x \in X(I),\quad \forall I \in \mathcal{I} \,\Bigr\}.
\end{equation}

\end{proposition}

\begin{proof}

Let $\gamma$ be any vector in $\R^k_+$. Consider the convex set
\begin{equation}\label{eq:B}
M_\gamma = \conv\bigl(\bigl\{\, f + \tfrac{r^j}{\gamma_j} \bigst \gamma_j > 0\,\bigr\}\bigr) + \cone\bigl(\bigl\{\,r^j \bigst \gamma_j =
0\,\bigr\}\bigr).  
\end{equation}
Since $\cone(\{r^1, \ldots, r^k\}) = \R^m$, we have that $f$ is in the
interior of $M_\gamma$. Observe that $\gamma_j =
\psi_{M_\gamma - f}(r^j)$. Using Theorem~\ref{thm:int_cuts}, it can be shown that $\sum_{i=1}^k\gamma_is_i
\geq 1$ is a valid inequality if and only if $M_\gamma$ does not have any integer
point in its interior. We denote the right hand side of \eqref{eq:B(Rf)}
by $$\Gamma=\Bigl\{\, \gamma \geq 0 \Bigst \sum_{j\in I} \gamma_j s_j(x,I)
\geq 1  \quad \forall x \in X(I),\quad \forall I \in \mathcal{I} \,\Bigr\}.$$ 

We first show that any $\gamma\in \Gamma$ gives the coefficients of a valid inequality. We will show that $M_\gamma$ does not contain any integer point in its interior. Suppose to the contrary and let $\bar x$ be a point in the interior of $M_\gamma$. If $\bar x - f \in \rec(M_\gamma)$, where $\rec$ denotes the recession cone, then $\bar x - f \in \cone\{\,r^j \st \gamma_j = 0\,\}$. Carath\'eodory's theorem for cones then implies that there exists a subset $I$ of $\{\,j \st \gamma_j = 0\,\}$ of size $m$ such that $\bar x - f \in \cone\{\,r^j \st j \in I\,\}$ and therefore $\bar x \in X(I)$. But then $\sum_{j \in I}\gamma_j s_j(\bar x, I) = 0 < 1$, which violates the inequality corresponding to $I$ and $\bar x$ in the definition of $\Gamma$. If $\bar x - f \not\in \rec(M_\gamma)$, then there exists $\mu > 1$ such that $\mu (\bar x - f) + f$ is on the boundary of $M_\gamma$ because $\bar x$ is in the interior of~$M_\gamma$.
This implies that $\mu (\bar x - f) + f$ lies on a facet of $M_\gamma$ and therefore, using Lemma~\ref{lem:caratheo}, there exists a subset $I$ of $\{\,j \st \gamma_j > 0\,\}$ and a subset $J$ of $\{\,j \st \gamma_j = 0\,\}$, with $\mu (\bar x - f) + f \in \conv(\{\,f + \frac{r^j}{\gamma_j}\st j \in I\,\}) + \cone(\{\,r^j\st j \in J\,\})$ and $|I| + |J|$ is at most $m$. Since the number of rays is at least $m+1$, we may assume that $|I| + |J| = m$. Without loss of generality, let us assume that $I = \{1, \ldots, |I|\}$ and $J = \{|I|+1, \ldots, m\}$. This then implies that there exist $\lambda_1 \geq 0, \ldots, \lambda_m \geq 0$ satisfying $\sum_{j=1}^{|I|} \lambda_j = 1$ and
\begin{align*}
\mu (\bar x - f) + f &= \textstyle \sum_{j=1}^{|I|} \lambda_j (f + \frac{r^j}{\gamma_j})
+ \sum_{j=|I|+1}^m \lambda_j r^j, \\
\intertext{thus}
\mu (\bar x - f) &= \textstyle \sum_{j=1}^{|I|} \lambda_j (\frac{r^j}{\gamma_j}) +
\sum_{j=|I|+1}^m \lambda_j r^j, \\
\intertext{and finally}
\bar x - f &= \textstyle \sum_{j=1}^{|I|} (\lambda_j/\mu) (\frac{r^j}{\gamma_j}) + \sum_{j=|I|+1}^m (\lambda_j/\mu) r^j.
\end{align*}

The last equation shows that $\bar x \in X(I \cup J)$. Moreover, \smash{$s_j(\bar x,
I\cup J) = \frac{\lambda_j}{\mu\gamma_j}$} for $1 \leq j\leq |I|$ and \smash{$s_j(\bar
x, I\cup J) = \frac{\lambda_j}{\mu}$} for $|I| + 1 \leq j \leq m$. Substituting
into the left-hand side of the constraint for~$\Gamma$ corresponding to $I\cup J$ and $\bar x$, we get $$\textstyle\sum_{j=1}^{|I|} \gamma_j\cdot\frac{\lambda_j}{\mu\gamma_j} + \sum_{j=|I|+1}^m 0\cdot\frac{\lambda_j}{\mu} = \sum_{j=1}^{|I|} \frac{\lambda_j}{\mu} < 1.$$ The inequality follows from the fact that $\sum_{j=1}^{|I|} \lambda_j = 1$ and $\mu > 1$. Therefore this constraint is violated by $\gamma$. So we reach a contradiction. Hence we conclude that $\intr(M_\gamma) \cap \Z^m = \emptyset$.

\medskip
We now show that if $\sum_{j=1}^k \gamma_j s_j \geq 1$ is a valid inequality, then $\gamma \in \Gamma$. If not, there exists $I \in \mathcal{I}$ and $x \in X(I)$ such that $\sum_{j \in I} \gamma_j s_j(x,I) < 1$. Let $I_+$ be the set $\{\,j \in I \st \gamma_j > 0\,\}$ and $I_0 = I\setminus I_+$. By definition,

\[
\begin{array}{rl}
& x - f = \sum_{j \in I} s_j(x,I)r^j = \sum_{j\in I_+} \gamma_j s_j(x,I)\frac{r^j}{\gamma_j} + \sum_{j \in I_0} s_j(x,I)r^j.
\end{array}
\]
Thus, 
\[
\begin{array}{rl}
& x = \mu f + \sum_{j\in I_+} \gamma_j s_j(x,I)(f + \frac{r^j}{\gamma_j}) + \sum_{j \in I_0} s_j(x, I)r^j,
\end{array}
\]
where $\mu = 1 - \sum_{j \in I} \gamma_j s_j(x,I) > 0$. Since $f \in
\intr(M_\gamma)$, the last equation shows that $x$ is in the interior of $M_\gamma$. This
contradicts the validity of $\sum_{j=1}^k \gamma_j s_j \geq 1$.
\end{proof}

The description of $\Blocking{\conv(\Rf)}$ in Proposition~\ref{prop:main} uses infinitely many inequalities. We now show that we need only finitely many of these inequalities. Given $I \in \mathcal{I}$, let $\ext(X(I))$ denote the extreme points of the convex hull of $X(I)$.

\begin{theorem}\label{thm:main}
$$\Blocking{\conv(\Rf)} = \Bigl\{\,\gamma \geq 0 \Bigst \sum_{j \in I} \gamma_j s_j(x,I)  \geq 1  \quad \forall x \in \ext(X(I)),\quad \forall I \in \mathcal{I}\,\Bigr\}.$$
\end{theorem}

\begin{proof}
We show that for any $I\in \mathcal{I}$ and $x \in X(I)$, the inequality $\sum_{j \in I} \gamma_j s_j(x,I)  \geq 1$ is dominated by a convex combination of inequalities corresponding to points in $\ext(X(I))$. Since $\{r^1, \ldots, r^k\}$ and $f$ are all rational, the recession cone of the convex hull of $X(I)$ is the same as $\cone(\{\,r^j\st j \in I\,\})$ (see, for example, Theorem~16.1 in~\cite{sch}). In fact, the convex hull of $X(I)$ is a polyhedron. Therefore, $x$ can be represented as $\sum_{p \in P} \mu_px_p + \sum_{j \in I} \lambda_j r^j$ where $x_p \in \ext(X(I))$ for all $p\in P$ and $\mu_p$ are convex coefficients and $\lambda_j$'s are nonnegative coefficients. This further implies that $x - f = \sum_{p \in P} \mu_p(x_p-f) + \sum_{j \in I} \lambda_j r^j$. 

If we represent $x-f$, $x_p - f$ in the basis $\{\,r^j \st j \in I\,\}$, we
conclude that $s_j(x, I) = \sum_{p\in P}\mu_ps_j(x_p, I) + \lambda_j$. Since
the $\lambda_j$'s are nonnegative, this shows that the inequality
corresponding to $x$ is dominated by a convex combination of the inequalities
corresponding to $x_p$, $p\in P$. 
\end{proof}

\subsection{Complexity of the inequality description of \boldmath$\Blocking{\conv(\Rf)}$}

We now turn to the study of the complexity of the inequality description of
the polyhedron~$\Blocking{\conv(\Rf)}$.\smallbreak

We use the following general result about the integer hull of a polyhedron.  
If $P$ is a polyhedron, we denote by $P_{\mathrm{I}}$ its integer hull, i.e., the convex hull of all integer points contained in $P$. When the dimension is fixed, $P_{\mathrm{I}}$ has only a polynomial number of
vertices, as Cook et al.~\cite{cook-hartmann-kannan-mcdiarmid-1992} showed.
\begin{theorem}
\label{THM:2ineqs}
 Let $P = \{\,  x\in\R^q \st A x\leq b\,\}$ be a
rational polyhedron
 with $A\in\Q^{p\times q}$ and let $\phi$ be the largest binary encoding size
 of any of the rows of the system~$A x\leq b$.  Let
$P_{\mathrm{I}} =
 \mathop{\mathrm{conv}}(P\cap\Z^q)$ be the integer hull of~$P$.  Then
the number of vertices
 of~$P_{\mathrm{I}}$ is at most $2 p^q{(6q^2\phi)}^{q-1}$.
\end{theorem}

Moreover, Hartmann \cite{hartmann-1989-thesis} gave an algorithm for
enumerating all the vertices, which runs in polynomial time in fixed
dimension.  
\smallbreak   

We thus obtain:

\begin{remark}\label{rem:poly_B(Rf)}
Let the dimension $m$ be a fixed number. Since all the rays $r^1, \ldots, r^k$
and $f$ are rational, by Theorem~\ref{THM:2ineqs}, the cardinality of
$\ext(X(I))$ is bounded by a polynomial in the binary encoding length of the
data $r^1, \ldots, r^k, f$ for any $I \in \mathcal{I}$. Moreover, the
cardinality of $\mathcal{I}$ is at most $k\choose m$, which is a polynomial in
$k$. Hence, $\Blocking{\conv(\Rf)}$ is a polyhedron which can be represented as the
intersection of polynomially many half-spaces. 
\end{remark}

\subsection{Finding the strongest cuts}

Let $\gamma^*$ be the optimal solution to the following convex program. 
\begin{equation}\label{eq:conv_prog}
\begin{aligned}
\min\ &\Vert \gamma \Vert_p \\
\text{s.t.}\ & \sum_{j \in I} \gamma_j s_j(x,I) \geq 1 & \quad \forall x \in \ext(X(I)),\quad \forall I \in \mathcal{I}, \\
& \gamma \geq 0.
\end{aligned}
\end{equation}

Theorem~\ref{thm:main} implies that $\gamma^*$ gives the coefficients of a valid inequality with minimum $l_p$ norm. There is an interesting interpretation for the optimal cut with respect to the $l_2$ norm. If we view \eqref{M_fk} as the optimal LP tableau, then valid inequalities for $\conv(\Rf)$ are cuts which separate the current LP solution, $x=f, s = 0$ from the integer hull. The valid inequality with minimum $l_2$ norm is then the ``deepest'' cut, i.e., the cut whose Euclidean distance from the current LP solution is the maximum. The other $l_p$ norms are also often used as a criterion for choosing the ``best'' cut.

\begin{remark} Since the feasible region for the convex
  program~\eqref{eq:conv_prog} is described by polynomially many inequalities
  by Remark~\ref{rem:poly_B(Rf)}, we can solve these programs in polynomial
  time. However, from a practical point of view, it might be easier to solve
  these programs using a cutting-plane or separation approach. We present a
  polynomial time separation algorithm for the convex program when the
  dimension~$m$ is an arbitrary fixed number, which uses integer feasibility
  algorithms in fixed dimensions. This avoids explicitly enumerating $I \in
  \mathcal{I}$ and $\ext(X(I))$, which could be a nontrivial and time-consuming task.

Given a point $\gamma$, we
need to decide if it is feasible for \eqref{eq:conv_prog}. This is achieved by
testing if the convex set $M_\gamma$ defined in \eqref{eq:B} has an integer point in
its interior. 

If $\smash{M_\gamma}$ is tested to have no integer point in its
interior, then Theorem~\ref{thm:int_cuts} implies that the inequality
$\sum_{j=1}^k\gamma_j s_j \geq 1$ is valid. The proof of
Proposition~\ref{prop:main} shows that $\gamma$ is therefore feasible to
\eqref{eq:conv_prog}. 

On the other hand, if $M_\gamma$ is tested to have an integer point $\bar x$ in
its interior, then the proof of Proposition~\ref{prop:main} shows that some
constraint corresponding to $I\in\mathcal{I}$ such that $\bar x \in X(I)$ is
violated. 

By testing each subset of $\{r^1, \ldots, r^k\}$ of size $m$, we can
find this violated constraint in $O(mk^m)$ calls to an integer feasibility
oracle. When $m$ is fixed, this is a polynomial in $k$.  

\end{remark}

Note that for the $l_1$ and $l_\infty$ norms, the optimization
  problem \eqref{eq:conv_prog} can be changed to a linear program by a
  standard reformulation.\smallbreak

Finding the valid inequality with minimum $l_\infty$ norm admits an alternative algorithm, which avoids solving
\eqref{eq:conv_prog}. This again utilizes only integer feasibility algorithms
for fixed dimensions. This approach could be more practical than solving the
linear program because it would avoid explicitly enumerating
$I \in \mathcal{I}$ and $\ext(X(I))$ and also does not require to use a
cutting-plane procedure.

Instead, we can use a simple search procedure as follows.
For any scalar $\alpha>0$, let $$C(\alpha) = \conv(\{\,f
+ \alpha r^j\st j = 1, \ldots, k\,\}).$$  

Let $\sum_{j=1}^k\gamma_j s_j \geq 1$ be a valid inequality. Let $M_\gamma$ be defined as in \eqref{eq:B}. Observe that $C(1/\Vert \gamma \Vert_\infty) \subseteq M_\gamma$. Since $M_\gamma$ does not contain any integer point in its interior, neither does $C(1/\Vert \gamma \Vert_\infty)$. Therefore, to find the inequality with optimal $l_\infty$ norm, we need to find the maximum
possible value of $\alpha$, such that $\intr(C(\alpha))\cap \Z^m = \emptyset$. Let this maximum be~$\alpha^*$. 

The maximum~$\alpha^*$, of course, corresponds to a set $C(\alpha^*)$ that has
an integer point on one of its facets.  This shows that $\alpha^*$ is a
rational number, for which, using standard techniques, we 
can determine a bound on its numerator and denominator of polynomial binary
encoding length. 

Then we can use the asymptotically optimal algorithm by Kwek and
Mehlhorn~\cite{Kwek2003-rational-search} for searching a rational number~$\alpha^*$ of
bounded numerator and denominator, using only queries of the type ``Is $\alpha^*
\leq \alpha$?''  This is similar to a binary search algorithm.  
Each such query amounts to testing $\intr(C(\alpha))\cap \Z^m = \emptyset$ for
some current estimate~$\alpha$ for $\alpha^*$. Thus, this query step can be
solved by integer feasibility algorithms for fixed dimensions. 



\section{The Tilting Space}\label{sec:tilting}





For any matrix $B = (b^1;\dots;b^n)\in \R^{n \times m}$, 
let $Y(B)$ be the set of integer points~$y^j$ contained in 
\begin{equation*}
\M(B) = \{\,x \in \R^m \st b^i\cdot (x-f) \leq 1\text{ for $i=1, \dots, n$}\,\}.
\end{equation*}
If $\M(B)$ is a lattice-free convex set, all elements of~$Y(B)$ of course lie
on the boundary of~$\M(B)$, that is, on at least one facet~$F_i$ of~$\M(B)$,
induced by a constraint $b^i\cdot(x-f)\leq 1$.  \smallbreak

In the present paper, we prove necessary conditions for $\sum_{j=1}^k
\psi_B(r^j)s_j \geq 1$ to be an extreme inequality mainly by perturbation
arguments.  Given a matrix~$B$, we show under suitable hypotheses the
existence of certain small perturbations $A$~and~$C$ of~$B$ such that 
the inequality $\sum_{j=1}^k
\psi_B(r^j)s_j \geq 1$ is a strict convex combination of the inequalities $\sum_{j=1}^k
\psi_A(r^j)s_j \geq 1$ and $\sum_{j=1}^k
\psi_C(r^j)s_j \geq 1$.
Geometrically, these perturbations correspond to slightly `tilting' the facets~$F_i$
of~$\M(B)$.
In our proofs, it is convenient to choose, for every $i=1, \dots, n$, 
a certain subset $Y_i\subseteq Y(B)\cap F_i$ of the integer points on the facet~$F_i$. 
When we tilt the facet~$F_i$, we require that this subset~$Y_i$ continues to lie in
the tilted facet; this obviously restricts how we can change the facet.  This
is illustrated in Figure~\ref{fig:tilting-intro}.
\begin{figure}[t]

\centering

\scalebox{0.85}{%
\ifpdf
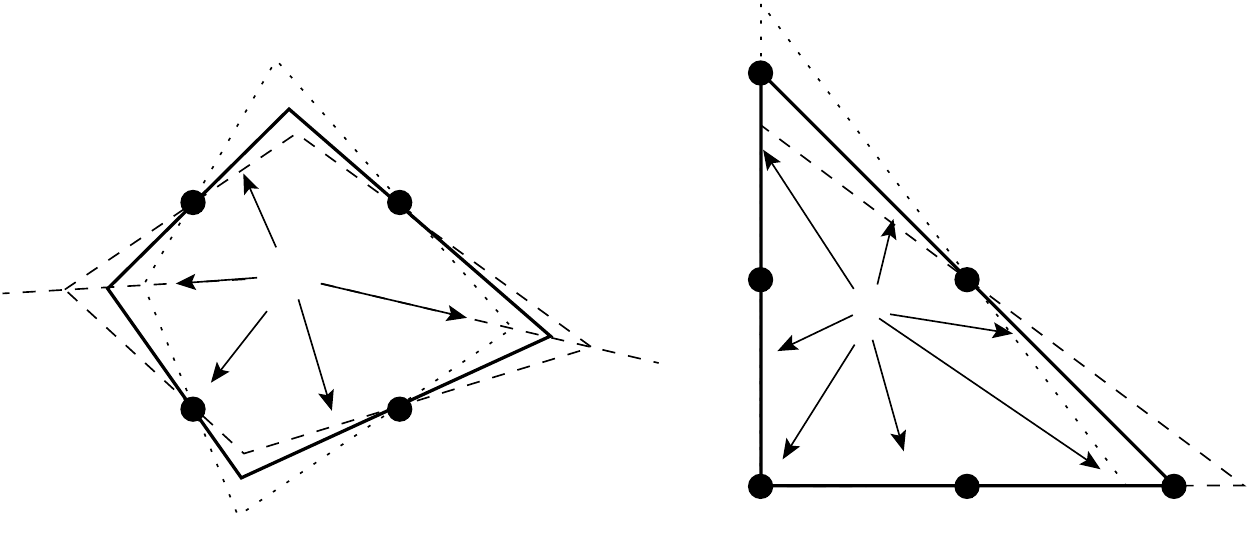
\else
\input{figures/figureQuadAndType1.pstex_t}
\fi}
%
%
%
  \caption{Tilting the facets of maximal lattice-free sets.
    (a)~In this particular quadrilateral, setting $Y_1=\{y^1\}, \dots, Y_4=\{y^4\}$
    allows to tilt all facets~$F_1,\dots,F_4$.  This still holds true if we
    ensure that all the corner rays remain corner rays for the perturbation (constraint~\eqref{eq:tilting-2}).
    (b)~In this Type-1 triangle, setting $Y_1=\{y^1\}$ (a strict subset of
    $Y(B) \cap F_1$) and $Y_2=Y(B)\cap F_2$, $Y_3=Y(B)\cap F_3$, then facet $F_1$
    can tilt, whereas facets $F_2$ and $F_3$ remain fixed.  This still holds
    true if we ensure that all the non-corner rays remain non-corner rays for
    the perturbation (constraint~\eqref{eq:tilting-3}).  Note that 
    choosing tilts from the set~$\mathcal S(B)$ ensures that no new integer
    points enter.  However, integer points may lie outside the set after
    tilting, such as the top and right vertices in this example.}
  \label{fig:tilting-intro}
\end{figure}

We also need to control the interaction of the rays~$r^j$ and the facets. We
will often refer to the set of \emph{ray intersections} $$ P = \bigl\{\,p^j \in \R^2\bigst
p^j = f + \tfrac{1}{\psi_B(r^j)} r^j ,\ \psi_B(r^j) > 0,\ j=1, \dots, k\,\bigr\},$$ 
that is, the points~$p^j$
where the rays~$r^j$ meet the boundary of the set~$\M(B)$.  

It is easy to see that  whenever
$\psi_B(r^j) > 0$, the set
$I_B(r^j) = \argmax_{i=1, \dots, n} b^i\cdot  r$ is the index set of all inequalities
of $\M(B)$ that the ray intersection $p^j = f + \frac{1}{\psi_B(r^j)} r^j$
satisfies with equality. 

In particular, for $m=2$, when all the inequalities corresponding to the rows
of $B$ are facets of $M(B)$, we have $|I_B(r^j)|= 1$ when $r^j$ points to the
relative interior of a facet, and $|I_B(r^j)|= 2$ when $r^j$ points to a vertex of
$\M(B)$. In this second case, we
call $r$ a \emph{corner ray} of $\M(B)$. Again see Figure~\ref{fig:tilting-intro}. When $M(B)$ is a split in $\R^2$, $|I_B(r^j)|=1$ if $r^j$ is {\em not} in the recession cone of $M(B)$ and $|I_B(r^j)|=2$ when $r^j$ is in the recession cone.
\begin{definition}
Let $\Y$ denote the tuple $(Y_1, \ldots, Y_n)$.  The \emph{tilting
  space}  $\T(B, \Y)\subset \R^{n\times m}$ is defined as  
the set of matrices $A = (a^1; \dots; a^n)\in\R^{n\times m}$ that satisfy the
following conditions:
\begin{subequations}
  \begin{align}
    a^i\cdot (y - f) &= 1 && \text{for} \ y \in Y_i,\ i=1, \dots, n,\label{eq:tilting-1}\\
    a^{i}\cdot  r^j &= a^{i'}\cdot  r^j   && \text{for} \ i,i' \in I_B(r^j),\label{eq:tilting-2}\\
    a^{i}\cdot r^j &> a^{i'}\cdot r^j && \text{for} \ i \in I_B(r^j),\ i'
    \notin I_B(r^j).\label{eq:tilting-3}
  \end{align}
\end{subequations}
\end{definition}
Constraint~\eqref{eq:tilting-2} implies that if $r^j$ hits a facet~$F_i$ of~$\M(B)$, then
it also needs to hit the same facet of~$\M(A)$.  In particular, for $m=2$, this
means that if $r^j$ is a corner ray of $\M(B)$, then $r^j$ must also be a
corner ray for $\M(A)$ if $A \in \T(B,\Y)$.  
Constraint~\eqref{eq:tilting-3} enforces that if $r^j$ does not hit a facet~$F_i$
of~$\M(B)$, then it also does not hit the same facet of~$\M(A)$. 
Thus we have $I_A(r^j) = I_B(r^j)$ for all rays $r^j$ if $A\in \T(B,\Y)$.

\smallbreak

Note that $\T(B,\Y)$ is cut out by linear equations and strict linear inequalities only and,
since we always have $B\in \T(B,\Y)$, it is non-empty.  Thus it is a convex set
whose dimension is the same as that of the affine space
defined by the equations, \eqref{eq:tilting-1}~and~\eqref{eq:tilting-2}, only. 
By $\mathcal N(B,\Y)\subset \R^{n\times m}$ we denote the linear space
parallel to this affine space, in other words the null space of these equations.

If $\dim \T(B, \Y) \geq 1$, we can find two other matrices $A$ and $C$ in $T(B,\Y)$
such that $B$ is a strict convex combination of $A$ and $C$. This will have the following important consequence which says that the inequality derived using $\M(B)$ is a convex combination of the inequalities derived using $\M(A)$ and $\M(C)$.

\begin{lemma}\label{lemma:convex_combination}
Suppose $A,C \in \T(B, \Y)$ with $B = \alpha A + (1-\alpha)C$, $\alpha \in (0,1)$.
  Then 
  $$\psi_B(r^j) = \alpha \psi_{A}(r^j) + (1-\alpha) \psi_{C}(r^j) \quad  \text{for}  \ j=1,\dots, k.$$
\end{lemma}
\begin{proof}
Let $j \in \{1,\dots ,k\}$.  Since $A, C \in \T(B, \Y)$ we know that $I_{B}(r^j) = I_{A}(r^j) = I_{C}(r^j)$.
Hence, let $i \in I_{B}(r^j)$.  Then  
\begin{align*}
\alpha \psi_{A}(r^j) + (1-\alpha) \psi_{C}(r^j) 
&= \alpha a^i\cdot  r^j + (1-\alpha) c^i\cdot   r^j\\
&= (\alpha a^i + (1-\alpha) c^i)\cdot   r^j
= b^i\cdot   r^j
= \psi_{B}(r^j).\tag*{\qed}
\end{align*}\pushQED\relax \end{proof}

Following the definition of extreme inequality, we see that finding such lattice-free polytopes $\M(A)$ and $\M(C)$ 
would imply that $\sum_{j=1}^k\psi_B(r^j)s_j \geq 1$ is not extreme provided that 
$\psi_{A}(r^j) \neq \psi_{C}(r^j)$ for some $j = 1, \dots, k$.  We will first handle the lattice-free condition, and later, via case analysis, we will argue that we can find distinct inequalities.\smallbreak

Next we introduce a tool that helps to ensure that no extra lattice points lie
in the set after tilting the facets.  To this end, 
consider the set
  $$\S(B) := \{\,A = (a^1; \dots; a^n) \in \R^{n \times m} \st Y(A) \subseteq Y(B)\,\}.$$   
\begin{lemma}\label{lemma: S(B) full-dimensional}
  Let $B\in \R^{n\times m}$ be such that $\M(B)$ is a bounded maximal lattice-free set.  Then $\S(B)$
  contains an open neighborhood of $B$ in the topology of $\R^{n\times m}$.
\end{lemma}
This follows from now-classic results in the theory of parametric linear
programming.  Specifically, consider a parametric linear program, $$\sup\{\,
c(t) x : A(t) x \leq b(t)\,\} \in\R\cup\{\pm\infty\},$$ 
where all coefficients
depend continuously on a parameter vector $t$ within some parameter region
$\mathcal R \subseteq \R^q$.  It 
is a theorem by D.~H.~Martin~\cite{martin-1975-continuity-maximum} that the optimal
value function is upper semicontinuous in every parameter point~$t_0$ such
that the solution set (optimal face) is bounded, relative to the set of parameters where the
supremum is finite.  Here we only make use of a lemma used in the proof:
\begin{theorem}[D.~H.~Martin~\cite{martin-1975-continuity-maximum}, Lemma~3.1]
  \label{thm: locally-bounded-solution-set}
  Suppose that the solution set for $t=t_0$ is non-empty and bounded.  Then,
  in parameter space, there is an open neighborhood~$\mathcal O$ of~$t_0$ such that the
  union of all solution sets for $t\in \mathcal O$ is bounded. 
\end{theorem}

\begin{proof}[Proof of Lemma~\ref{lemma: S(B) full-dimensional}]
  Consider the parametric linear program
  $$ \max \{\,0 \st a^i\cdot (x-f) \leq 1,\ i=1, \dots, n\,\}$$
  with parameters $t = A = (a^1;\dots;a^n) \in \R^{n\times m}$. 
  By the assumption of the lemma, the solution set for $t_0 = B = (b^1;\dots;b^n)$
  is bounded.  Let $\mathcal O$ be the open neighborhood of~$t_0$ from
  Theorem~\ref{thm: locally-bounded-solution-set},
  and let $\hat S$ be the union of all solution sets for $t\in \mathcal O$, which is by
  the theorem a bounded set.  
  
  For each of the finitely many lattice points $y \in \hat S \setminus \M(B)$, let 
  $i(y)\in\{1,\dots,n\}$ be an index of an inequality that cuts off~$y$, that
  is, $b^{i(y)}\cdot (y - f) > 1$.   
  Then $$\mathcal O' = \{\, A=(a^1;\dots;a^n) \in \mathcal O \st a^{i(y)}\cdot (y - f) > 1 \text{ for all $y \in
    \hat S \setminus \M(B)$} \,\}$$
  is an open set containing~$B=(b^1;\dots;b^n)$.
  For $A=(a^1;\dots;a^n) \in \mathcal O'$ we have $Y(A) \subseteq Y(B)$, and thus 
  $\mathcal O'$ is the desired open neighborhood of~$B$ contained in~$\S(B)$.
\end{proof}

\begin{observation}\label{obs:dimension}
Suppose $\dim\T(B, \Y) \geq 1$.  By virtue of Lemma~\ref{lemma: S(B)
  full-dimensional}, for any $\bar A\in\mathcal N(B,\mathcal Y)$, there exists $0 < \delta < 1$ 
such that both $B \pm \epsilon \bar A \in \T(B, \Y) \cap \S(B)$ for all $0 < \epsilon \leq \delta$.
\end{observation}

\begin{observation}\label{obs:latticefree}
If $\Y  = (Y_1, \dots, Y_n)$ is a covering of $Y(B)$, then $\M(A)$ is lattice-free for every $A \in \T(B, \Y) \cap \S(B)$.
\end{observation}

Observation \ref{obs:dimension} and \ref{obs:latticefree} are very useful
because when we can ensure that $\Y  = (Y_1, \dots, Y_n)$ is a covering of
$Y(B)$, we no longer have to worry about finding explicit lattice-free convex
sets. Rather, we can concentrate on simply showing that $\dim \T(B, \Y) \geq
1$ and that there exist matrices in that space such that there is a change in
the coefficient of at least one of the rays.



\smallbreak

\begin{figure}
\centering

\scalebox{0.85}{%
\ifpdf
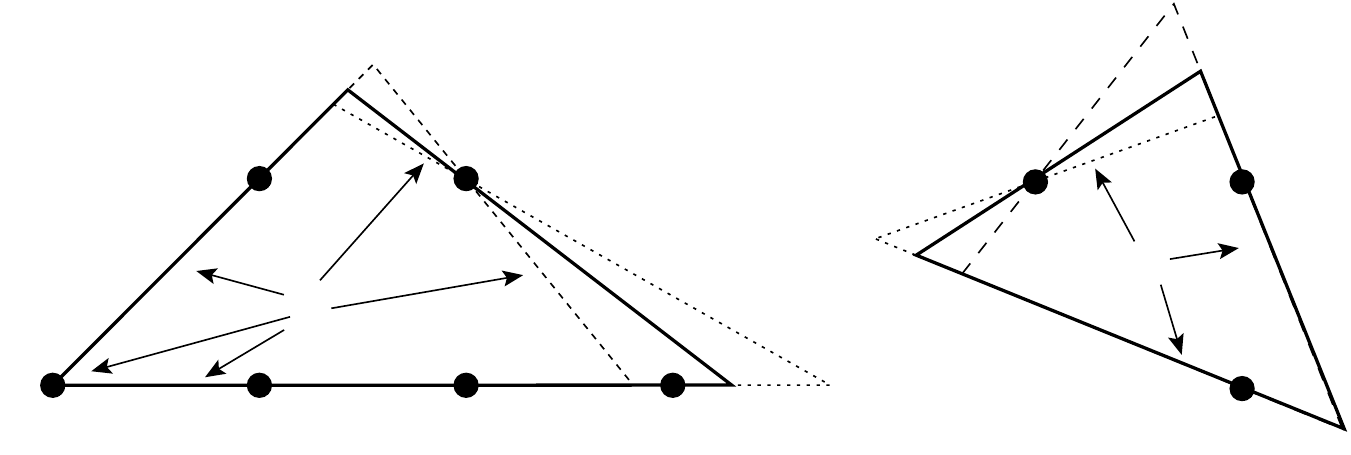
\else
\input{figures/figureSimpleTilts.pstex_t}
\fi}
\caption{Simple tilts: Tilting one facet of a polytope to generate new inequalities.  In
  both examples, there is a ray pointing to a non-integer point on the interior
  of the facet being tilted.  This ensures that the inequalities from the
  tilted sets are distinct, and therefore we see that the original inequality
  $\sum_{j=1}^k\psi_B(r^j)s_j\geq 1$ is not extreme because it is the strict convex combination of two
  other inequalities.  This is the assertion of
  Lemma~\ref{lemma:simple_tilts}.} 
\label{figure:simple}
\end{figure}
A simple application of this principle is to tilt one facet of a polytope to
show that the corresponding inequality is not extreme, as shown in Figure
\ref{figure:simple}.  This is summarized in the following lemma. 
\begin{lemma}[Simple tilts]\label{lemma:simple_tilts}
Let $m\geq 2$. Let $\M(B)$ be a maximal lattice-free polytope for some matrix $B\in \R^{n\times m}$. Let $F_1$ be a facet of $\M(B)$ such that $\relint(F_1) \cap \Z^m = \{y^1\}$ and $P \cap F_1 \subset \relint(F_1)$, i.e., there are no ray intersections on the lower-dimensional faces of $F_1$.  If $\relint(F_1) \cap P \setminus \Z^m \neq \emptyset$, then $\sum_{j=1}^k\psi_B(r^j)s_j\geq 1$ is not extreme.  
\end{lemma}
\begin{proof}
Let $F_1,\dots,F_n$ be the facets of $M(B)$. Let $Y_1 = \{y^1\}$ and $Y_i=Y(B)\cap F_i$, $i = 2, \ldots, n$, so 
that $\Y=(Y_1,\dots,Y_n)$ is a covering of the set $Y(B)$ of integer points in~$\M(B)$.  

Let us analyze $\dim \T(B, \Y)$. Since $P \cap F_1 \subset \relint(F_1)$, there are no equalities in $\T(B, \Y)$ corresponding to some $I_B(r^j)$ which involve $a^1$. Moreover, $Y_1$ is a singleton set consisting of $y^1$. Hence, there is only one equation in $\T(B, \Y)$ which involves $a^1$, and that is $a^1 \cdot (y^1-f) = 1$. This implies that $\dim\T(B,\Y) \geq m-1 \geq 1$ for $m \geq 2$. We will now select a particular element in $\mathcal{N}(B, \Y)\setminus\{0\}$.

By the hypothesis, there exists $j\in \{1, \ldots, k\}$ such that $p^j \in (\relint(F_1) \cap P) \setminus \Z^m$. Since $\relint(F_1)\cap\Z^m = \{y^1\}$, this implies $r^j$ and $y^1-f$ are linearly independent.
Since $a^1 \cdot (y^1-f) = 0$ is the only equation involving $a^1$ in
$\mathcal{N}(B, \Y)$, and $y^1-f$ and $r^j$ are linearly independent, $\mathcal
N(B, \Y) \cap \{\,(a^1; \ldots ; a^n) \st a^1\cdot r^j = 0 \,\} \subsetneq
\mathcal{N}(B, \Y)$. Pick any $\bar A \in \mathcal N(B, \Y) \setminus
\{\,(a^1; \ldots ; a^n)\st a^1\cdot r^j = 0 \,\}$.

By
Observation~\ref{obs:dimension},  there exists an $\epsilon > 0$  such that
both $B \pm \epsilon \bar A \in \T(B, \Y) \cap \S(B)$. By our choice of $\Y$,
the hypothesis of Observation~\ref{obs:latticefree} is satisfied and therefore
$\M(B \pm \epsilon \bar A)$ are both lattice-free. Moreover, since $\bar A \not\in
\{\,(a^1; \ldots ; a^n) \st a^1\cdot r^j = 0 \,\}$, we have $\bar a^1\cdot r^j
\neq 0$. Therefore, $\psi_{B + \epsilon\bar A}(r^j) = (b^1  +\epsilon \bar a^1)\cdot r^j \neq (b^1  - \epsilon \bar a^1)\cdot r^j = \psi_{B - \epsilon\bar A}(r^j)$; the equalities follow from the fact that $B \pm \epsilon A \in \T(B,\Y)$ and so $I_{B+\epsilon\bar A}(r^j) = I_{B-\epsilon\bar A}(r^j) = I_B(r^j) = \{1\}$. Moreover, since $B = \frac{1}{2}(B + \epsilon \bar A) + \frac{1}{2}(B - \epsilon \bar A)$, one can now apply Lemma~\ref{lemma:convex_combination} to
show that the inequality from $M(B)$ is a convex combination of the
two different valid inequalities coming from $M(B\pm \epsilon \bar A)$. 
\end{proof}
In the next section, we will use this lemma and more complicated applications
of the tilting space.

\section{New Necessary Conditions for \boldmath$m=2$}\label{sec:nec_cond}

In this section, we prove necessary conditions for $\sum_{j=1}^k\psi_B(r^j)s_j\geq 1$ to be an extreme
inequality for any matrix $B$ such that $\M(B)$ is a maximal lattice-free set
in $\R^2$.  These conditions can also be shown using the complete
characterization of facets for $m=2$ in \cite{cm}. Our proofs primarily use
geometrically motivated tilting arguments which illuminate why certain
inequalities are not extreme.  

We find only three cases when a non-extreme
inequality is a convex combination of inequalities derived from convex sets of
a different combinatorial type: splits can be convex combinations of two
Type~2 triangle inequalities; Type~2 triangles can, in some instances, be
convex combinations of a Type~3 triangle and a quadrilateral inequality; and
in some other cases, Type~2 inequalities can be convex combinations of two
quadrilaterals.  In Section~\ref{sec:poly_facets}, we will use these
conditions to show that there are only polynomially many extreme inequalities
for $\conv (\Rf)$.

\paragraph{Notation.}  
The integer points will typically be labeled such that $y^1 \in \relint(F_1), y^2 \in \relint(F_2)$.  The closed line segment between two points $x^1$ and $x^2$ will be denoted by $[x^1, x^2]$, and the open line segment will be denoted by $(x^1, x^2)$.
Within the case analysis of some of the proofs, we will refer to certain
points lying within splits.  For convenience, for $i=1,2,3$, we define $S_i$ as the split such that one facet of $S_i$ contains~$F_i$ and $S_i\cap \intr(\M(B)) \neq \emptyset$. For any facet $F_i$, we will need to consider the sub-lattice of~$\Z^2$ contained in the linear space parallel to $F_i$. We use the notation $v(F_i)$ to denote the primitive lattice vector which generates this one-dimensional lattice. \bigbreak

We begin with a lemma regarding corner rays for triangles and quadrilaterals in $\R^2$.

\begin{lemma}
\label{lemma:corner_rays_almost}
Let $B \in \R^{n \times 2}$ be such that $M(B)$ is a triangle ($n=3$) or a
quadrilateral ($n=4$). Let $Y_i = \{y^i\}$, for any $y^i \in \relint(F_i) \cap
\Z^2$.  If $P \not\subset \Z^2$ and $M(B)$ has fewer than $n$ corner rays,
then there exists $\bar A \in \mathcal N(B, \Y) \setminus\{0\}$ such that for all $0 < \epsilon < 1$
$\psi_{B+\epsilon \bar A}(r^j) \neq \psi_{B-\epsilon \bar A}(r^j)$ for some $j = 1, \dots, k$ and $\psi_B(r^j)
= \frac{1}{2} \psi_{B-\epsilon \bar A}(r^j) +\frac{1}{2} \psi_{B+\epsilon \bar A}(r^j)$ for all $j = 1, \ldots, k$. 
\end{lemma}
\begin{proof}
We examine the tilting space of $B$ with at most $n-1$ corner rays.  We only
need to examine the tilting space of exactly $n-1$ corner rays, as it is a
subspace of the other cases.  With $n-1$ corner rays, $\T(B,\Y)$ is the set of
matrices $A=(a^1;\dots;a^n)$ satisfying the following system of equations,
where, for convenience, we define $\y^i:= y^i - f$:
\begin{equation*}
a^i \cdot \y^i = 1  \ \text{for} \ i=1,\dots,n  \hspace{.5cm} \text{ and } \hspace{.5cm}    a^i \cdot r^i = a^{i+1} \cdot r^i \ \text{for} \  i = 1,\dots, n-1,
\end{equation*}
and a number of strict inequalities, which we do not list here.

We have assumed, without loss of generality, that the rays and facets are
numbered such that we have corner rays $r^i \in F_i \cap F_{i+1}$ for
$i=1,\dots, n-1$, 
so the remaining ray~$r^n$ is not a corner ray.
As usual,
$y^i \in F_i\cap \Z^2$ for $i=1,\dots, n$.  Note that $\bar y^i$ is linearly
independent from $r^i$ for $i=1, \dots, n-1$ and linearly independent from
$r^{i-1}$ for $i=2, \dots, n$, because $y^i$ lies in the relative interior of $F_i$ and the rays point to the vertices.   

We now study the linear subspace $\mathcal N(B,\Y)$ that lies parallel to the affine hull of
$\T(B,\Y)$, so that $\mathcal N(B,\Y)$ is described by the homogeneous equations  
\begin{equation}\label{eq:N(B,Y)}
a^i \cdot \y^i = 0  \ \text{for} \ i=1,\dots,n  \hspace{.5cm} \text{ and } \hspace{.5cm}    a^i \cdot r^i = a^{i+1} \cdot r^i \ \text{for} \  i = 1,\dots, n-1.
\end{equation}
There are $2n-1$ equations and $2n$ variables, so $\dim  \mathcal N(B,\Y) \geq 1$. Moreover, observe that $B$ satisfies all the strict inequalties of $\T(B,\Y)$ and therefore, we can choose $\bar A = (\a^1; \dots ; \a^n) \in \mathcal N(B,\Y)\setminus\{0\}$ such that $B\pm \epsilon\bar A \in \T(B,\Y)$ for all $0 < \epsilon < 1$.

Notice that for $i=1,\dots, n-1$, if $\a^i = 0$, then $\a^{i+1}$ must satisfy $\a^{i+1} \cdot r^i =0$ and $\a^{i+1} \cdot \y^{i+1} = 0$, which implies that $\a^{i+1} = 0$, since $\bar y^{i+1}$ and $r^i$ are linearly independent.  
Similarly, for $i=2,\dots, n$, if $\a^i = 0$, then $\a^{i-1}$ must satisfy $\a^{i-1} \cdot r^{i-1} = 0$ and $\a^{i-1} \cdot \y^{i-1} = 0$, which implies that $\a^{i-1} = 0$.  
By induction, this shows that if $\a^i = 0$ for any $i=1,\dots, n$, then $\bar A = 0$, which contradicts our assumption.  Hence, $a^i \neq 0$ for $i=1,\dots, n$.

Now suppose the ray $r \in \{r^1, \dots, r^k\}$ points to $F_i\setminus \Z^2$ for some $i \in \{1, \dots, n\}$.  
This ray must exist by the assumption that $P \not\subset \Z^2$.  If $r$ is
parallel to~$\y^i$, then it either points to~$y^i$ from~$f$, or it does not point to~$F_i$.  
Since we assumed that $r$ points to $F_i \setminus \Z^2$, neither of these is possible, so $r$ is not parallel to $\y^i$.  
Now since $\a^i \cdot \y^i = 0$ and neither is the zero vector, $\y^i$ and
$\a^i$ are linearly independent and thus span $\R^2$.  Pick $\alpha, \beta$ such that $r = \alpha \y^i + \beta \a^i$.  
Then $\a^i \cdot r = \a^i\cdot (\alpha \y^i + \beta \a^i)  = \beta \Vert\a^i\Vert^2_2$.  
Note $\beta\neq 0$ since $r$~is not parallel to~$\y^i$. 
Since $B\pm \epsilon \bar A \in \T(B,\Y)$ for every $0 < \epsilon < 1$, $I_{B + \epsilon \bar A}(r) = I_B(r) = I_{B - \epsilon \bar A}(r)$. Therefore, $\psi_{B + \epsilon\bar{A}}(r) = (b^i + \epsilon \bar a^i)\cdot r \neq (b^i - \epsilon\bar a^i)\cdot r = \psi_{B - \epsilon\bar{A}}(r)$.  
Since $B = \frac{1}{2}(B+\epsilon \bar A) + \frac{1}{2}(B-\epsilon \bar A)$, applying Lemma \ref{lemma:convex_combination} finishes the result.
\end{proof}

We comment here that in the statement of Lemma~\ref{lemma:corner_rays_almost}, we do not insist that $M(B)$ is a lattice-free convex set. Therefore, the statement does not mention anything about valid or extreme inequalities for $\conv(\Rf)$. This generality will be needed in our results in the coming subsections.

\subsection{Type 3 triangles and quadrilaterals}

For this section on Type 3 triangles and quadrilaterals, we will be using a specific $\Y = (Y_1, \ldots, Y_n)$ where $Y_i$ will consist of the unique integer point in the relative interior of facet~$F_i$. This would mean that $\Y = (Y_1, \ldots, Y_n)$ is a covering of $Y(B)$ for Type 3 triangles and quadrilaterals. We will now apply Lemma~\ref{lemma:corner_rays_almost} to matrices $B$ such that $M(B)$ is a maximal lattice-free set that is either a Type 3 triangle or a quadrilateral.

\begin{corollary}
\label{lemma:corner_rays}
Suppose that $\M(B)$ has $n$ facets and is  a maximal lattice-free set that is either a Type 3 triangle $(n=3)$ or a quadrilateral $(n=4)$, and that $P \not\subset \Z^2$.  If $M(B)$ has fewer than $n$ corner rays, then $\sum_{j=1}^k\psi_B(r^j)s_j\geq 1$ is not extreme. 
\end{corollary}
\begin{proof}
Apply Lemma~\ref{lemma:corner_rays_almost} on $\M(B)$ with $\Y$ to obtain $\bar A \in \mathcal{N}(B,\Y)\setminus\{0\}$ with the stated properties. Since $\Y$ is a covering of $Y(B)$, by Observation~\ref{obs:dimension}, there exists $0 < \epsilon < 1$ such that $B\pm \epsilon \bar A \in \T(B,\Y)\cap S(B)$; so by Observation~\ref{obs:latticefree}, $M(B \pm \epsilon \bar A)$ are both lattice-free.  From the conclusion of Lemma \ref{lemma:corner_rays_almost}, we see that $\sum_{j=1}^k\psi_B(r^j)s_j\geq 1$ is not extreme as it is the convex combination of two distinct valid inequalities derived from the lattice-free sets $M(B \pm \epsilon \bar A)$.
\end{proof}

\begin{lemma}[Type 3 Triangles]\label{lemma:type 3 conditions}
Suppose $\M(B)$ is a Type 3 triangle. If $\sum_{j=1}^k\psi_B(r^j)s_j\geq 1$ is extreme, then one of the following holds: \\
\textbf{Case a.} $P \subset \Z^2$.\\
\textbf{Case b.} $\verts(B) \subseteq P$.
\end{lemma}
\begin{proof}
This follows from Corollary \ref{lemma:corner_rays}.
\end{proof}

For quadrilaterals, Cornu\'ejols and Margot defined the \emph{ratio condition}
as a necessary and sufficient condition to yield an extreme inequality when all corner rays
are present.  Suppose $p^1, p^2, p^3, p^4$ are the corner ray intersections
assigned in a counter-clockwise orientation, and $y^i$ is the integer point
contained in $[p^i, p^{i+1}]$.  The ratio condition holds if there does not
exist a scalar $t > 0$ such that 
\begin{equation}\label{eq:ratio-cond}
\frac{\Vert y^i - p^i\Vert}{\Vert y^i - p^{i+1}\Vert} = \begin{cases}
t & \text{for} \ i = 1,3\\
\frac{1}{t} & \text{for} \ i = 2,4.
\end{cases}
\end{equation}
This is illustrated in Figure~\ref{fig:ratio-cond}.
We will now show the relation between the ratio condition and the tilting space.
\begin{figure}
\centering

\scalebox{0.85}{%
\ifpdf
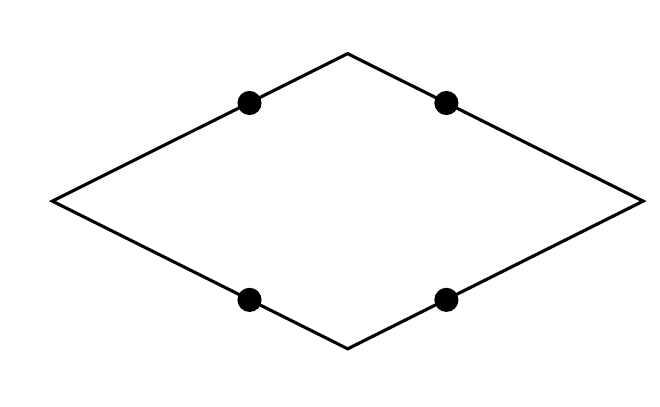
\else
\input{figures/figureRatioCondition.pstex_t}
\fi}
\caption{Example of a quadrilateral for which the ratio condition does \emph{not}
  hold, i.e., there exists a $t>0$ satisfying \eqref{eq:ratio-cond}. Here
  $\dim \T(B,\Y) \neq 0$.}
\label{fig:ratio-cond}
\end{figure}
\begin{lemma}
\label{lemma:ratio_condition}
Suppose $\M(B)$ is a quadrilateral with four corner rays.  If the ratio
condition does not hold, i.e., there exists a scalar $t>0$
with~\eqref{eq:ratio-cond}, then  $\dim \T(B,\Y) \neq 0$.
\end{lemma}
\begin{proof}

We will first analyze the tilting space equations with four corner rays, and then apply the assumption that the ratio condition does not hold.
For convenience we define $\y^i:= y^i - f$ and $\p^i := p^i - f$, where $p^i
$ are the ray intersections.  Then $\p^i = 
\frac{1}{\psi_B(r^i)} r^i$. 

We want to determine when there is not a unique solution to the following system of equations that come from the tilting space:
\begin{equation*}
\begin{array}{ccc}
\begin{array}{r@{\;}l}
a^1 \cdot \y^1&= 1\\
a^1 \cdot \p^2&= a^2\cdot  \p^2 \\
a^2\cdot \y^2  &= 1\\
a^2\cdot \p^3  &= a^3\cdot \p^3 \\
a^3\cdot \y^3  &= 1\\
a^3\cdot \p^4  &=a^4\cdot  \p^4 \\
a^4 \cdot \y^4 &= 1\\
a^4\cdot \p^1  &=a^1\cdot  \p^1 
\end{array}
&
\quad\text{or}\quad
&
\begin{bmatrix}
\y^1 \\
\p^2 & -\p^2 \\
& \y^2\\
& \p^3 & - \p^3 \\
& & \y^3\\
 & & \p^4 & -\p^4\\
 & & & \y^4\\
 - \p^1 & & & \p^1
\end{bmatrix}
\begin{bmatrix}
a^1\\ a^2 \\ a^3\\a^4
\end{bmatrix}
= 
\begin{bmatrix}
1\\0\\1\\0\\1\\0\\1\\0
\end{bmatrix}
\end{array}
\end{equation*}
as an $8 \times 8$  matrix equation where every vector shown in the matrix is a row vector of
size 2.  We will analyze the determinant of the matrix.

Since the points $\y^1, \y^2, \y^3, \y^4$ are on the interior of each facet,
they can be written as certain convex combinations of $\p^1, \p^2, \p^3, \p^4$.  We write this in a complicated form at first to simplify resulting calculations.  Here, $\alpha' = 1 + \alpha$, and $\alpha >0$, and similarly for $\beta, \gamma$, and~$\delta$.
$$
\begin{array}{r@{\;}lcr@{\;}l}
\y^1 &= \frac{1}{\alpha'} \p^1 + \frac{\alpha}{\alpha'} \p^2 & &\p^1 &= \alpha' \y^1 - \alpha \p^2\\
\y^2 &= \frac{1}{\beta'} \p^2 + \frac{\beta}{\beta'} \p^3 &\Leftrightarrow &\p^2 &= \beta' \y^2 - \beta \p^3\\
\y^3 &= \frac{1}{\gamma'} \p^3 + \frac{\gamma}{\gamma'} \p^4 & &\p^3 &= \gamma' \y^3 - \gamma \p^4\\
\y^4 &=\frac{1}{\delta'} \p^4 + \frac{\delta}{\delta'} \p^1 & &\p^4 &= \delta' \y^4 - \delta \p^1
\end{array}
$$
Now just changing the last row using the above columns
$$
\begin{bmatrix}
 - \p^1 & 0& 0& \p^1
\end{bmatrix} \rightarrow
\begin{bmatrix}
  0&\alpha \p^2 &0 & \p^1
\end{bmatrix} \rightarrow
\begin{bmatrix}
 0 &0 &-\alpha \beta \p^3 & \p^1
\end{bmatrix} \rightarrow
\begin{bmatrix}
0  &0 &0 & \alpha \beta \gamma \p^4 + \p^1 
\end{bmatrix}
$$
The resulting matrix, after adding this last row and substituting in $\y^4$, is
$$
\begin{bmatrix}
\y^1 \\
\p^2 & -\p^2 \\
& \y^2\\
& \p^3 & - \p^3 \\
& & \y^3\\
 & & \p^4 & -\p^4\\
 & & & \frac{1}{\delta'} \p^4 + \frac{\delta}{\delta'} \p^1\\
  & & & \alpha \beta \gamma \p^4 + \p^1 
\end{bmatrix}
$$
This is now an upper block triangular matrix.  The first three blocks are all non-singular, and the last block is non-singular if and only if there does not exist a $t$ such that 
$$
  \frac{1}{\delta'} \p^4 + \frac{\delta}{\delta'} \p^1 =  t(\alpha \beta \gamma \p^4 + \p^1 )\ \Rightarrow \ \Big(\frac{\delta}{\delta'} - t\Big) \p^1 + \Big(\frac{1}{\delta'} - t \alpha \beta \gamma\Big)\p^4 = 0.
$$
If such a $t$ exists, then $t = \frac{\delta}{\delta'}$ since $\p^1$ and $\p^4$ are linearly independent.    It follows that $\alpha \beta \gamma \delta = 1$ if and only if $\dim \T(B,\Y) \neq 0$.  If the ratio condition does not hold, then it is easy to see that $\alpha= \frac{1}{\beta} = \gamma  = \frac{1}{\delta}$, and hence $\alpha \beta \gamma \delta = 1$ and $\dim \T(B,\Y) \neq 0$.
\end{proof}

\begin{lemma}[Quadrilaterals]\label{lemma:quadrilateral cond}
Suppose $\M(B)$ is a quadrilateral.  If $\sum_{j=1}^k\psi_B(r^j)s_j\geq 1$ is extreme, then one of the following holds: \\
\textbf{Case a.} $P \subset \Z^2$. \\
\textbf{Case b.} $\verts(B) \subseteq P$ and the ratio condition holds. Moreover, $\M(B)$ is the unique quadrilateral with these four corner rays and these four integer points.  

\end{lemma}
\begin{proof}
Suppose that we are not in Case a.  Corollary \ref{lemma:corner_rays} shows that all four corner rays must exist.  Lemma \ref{lemma:ratio_condition} shows that if the ratio condition does not hold, then $\dim\T(B,\Y) \geq 1$ and so one of the equalities in $\T(B,\Y)$ corresponding to a corner ray is redundant. This means that $N$ is a subspace of $\mathcal{N}(B,\Y)$ where $N$ is the subspace given by the equations \eqref{eq:N(B,Y)}. Since we suppose $P \not\subset \Z^2$, the proof of Lemma \ref{lemma:corner_rays_almost} shows that there exists $\bar A\in N\setminus\{0\}$ such that for every $0 < \epsilon < 1$, $\psi_{B+\epsilon \bar A}(r^j) \neq \psi_{B-\epsilon \bar A}(r^j)$ for some $j = 1, \dots, k$ and $\psi_B(r^j) = \frac{1}{2} \psi_{B-\epsilon \bar A}(r^j) +\frac{1}{2} \psi_{B+\epsilon \bar A}(r^j)$ for all $j = 1, \ldots, k$. Since $N$ is a subspace of $\mathcal{N}(B,\Y)$, we have that $\bar A \in \mathcal{N}(B,\Y)\setminus\{0\}$. We can again use Observations~\ref{obs:dimension} and \ref{obs:latticefree} to show that $\sum_{j=1}^k\psi_B(r^j)s_j\geq 1$ is not extreme.  

Observe that the set of matrices $A$ such that $M(A)$ contains the same set of integer points as $M(B)$ and has the same four corner rays as $M(B)$ is given by all solutions to the equality system in $\T(B, \Y)$. If this system had non unique solutions, then $\dim\T(B,\Y) \geq 1$ and following the same reasoning as above, we would conclude that $\sum_{j=1}^k\psi_B(r^j)s_j\geq 1$ is not extreme. 
\end{proof}

\begin{remark}
The ratio condition is indeed equivalent to $\dim \T(B,\Y) = 0$.  We can see this by showing that $\dim \T(B,\Y) \neq 0$ if and only if the ratio condition does not hold.  Lemma \ref{lemma:ratio_condition} shows that if the ratio condition does not hold, then $\dim \T(B,\Y) \neq 0$.  On the other hand, if $\dim \T(B,\Y) \neq 0$, then $\sum_{j=1}^k\psi_B(r^j)s_j\geq 1$ is not extreme using similar arguments as in the proof above of Lemma~\ref{lemma:quadrilateral cond}. Cornu\'ejols and Margot \cite{cm} show that the ratio condition holds if and only if $\sum_{j=1}^k\psi_B(r^j)s_j\geq 1$ is extreme, and so since $\sum_{j=1}^k\psi_B(r^j)s_j\geq 1$ is not extreme, the ratio condition does not hold.
\end{remark}

\subsection{Type 1 triangles}

\begin{figure}%
\centering

\scalebox{0.85}{%
\ifpdf
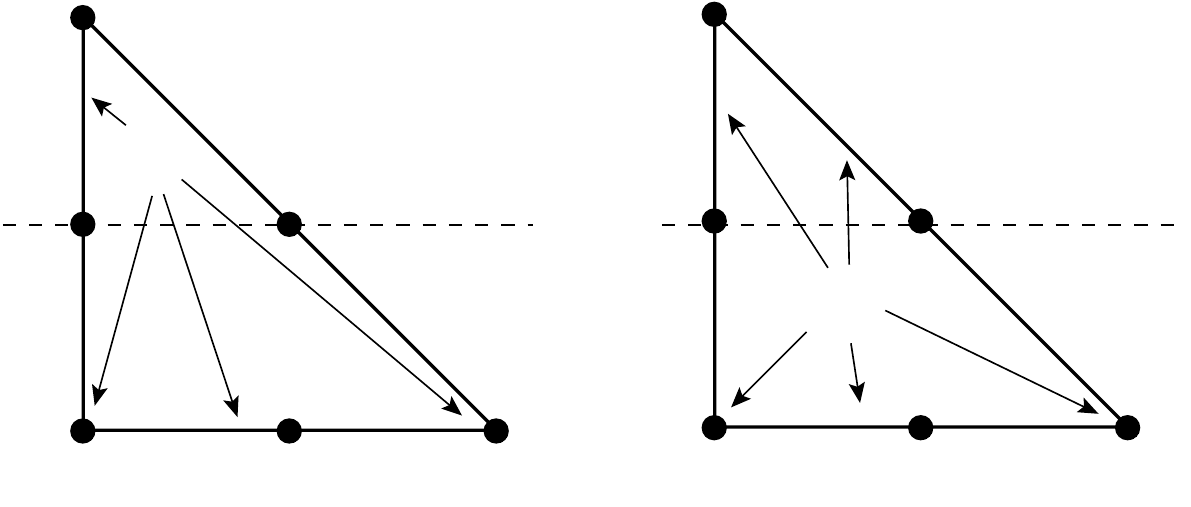
\else
\input{figures/figureType1Cases.pstex_t}
\fi}
\caption{Cases for extreme Type 1 triangles in Lemma~\ref{lem:Type1}}
\label{fig:type1}
\end{figure}%

\begin{lemma}[Type 1 Triangles]
\label{lem:Type1}
Suppose $\M(B)$ is a Type 1 triangle and suppose that $\sum_{j=1}^k\psi_B(r^j)s_j\geq 1$ cannot be realized or dominated by an inequality derived from either a Type 2 triangle or a split. 
 If $\sum_{j=1}^k\psi_B(r^j)s_j\geq 1$ is extreme, then there exist $p^1, p^2 \in \verts(B) \cap
 P$. Moreover, labeling the facet containing $p^1, p^2$ as~$F_3$, one of the following holds:\\
\textbf{Case a.} $f \notin S_3$.\\
\textbf{Case b.} $f \in S_3$, and $P \not\subset S_3$.
\end{lemma}%

Figure~\ref{fig:type1} illustrates the two cases of the lemma.

\begin{proof}
\textbf{Step 1.}  We will show that if $\#(\verts(B) \cap P) \leq 1$, then either $\sum_{j=1}^k\psi_B(r^j)s_j\geq 1$ is not extreme, or it is realized by a Type 2 inequality.

If $\#(\verts(B) \cap P) \leq 1$, then there is a facet whose vertices are not
contained in $P$; without loss of generality, let this facet be $F_1$.  We now
consider a simple tilt of facet $F_1$.  Lemma \ref{lemma:simple_tilts} shows
that if $P \cap \relint(F_1) \setminus\Z^2 \neq \emptyset$, then $\sum_{j=1}^k\psi_B(r^j)s_j\geq 1$ is
not extreme.  Otherwise, if $P \cap \relint(F_1) \setminus\Z^2 =
\emptyset$, then since there are no corner rays, we can tilt $F_1$ with $y^1$
as a fulcrum and create a Type~2 triangle that realizes the same inequality as
$\sum_{j=1}^k\psi_B(r^j)s_j\geq 1$ (see Figure~\ref{figure:simple2}).

%

\smallbreak
\begin{figure}
\centering

\scalebox{0.85}{%
\ifpdf
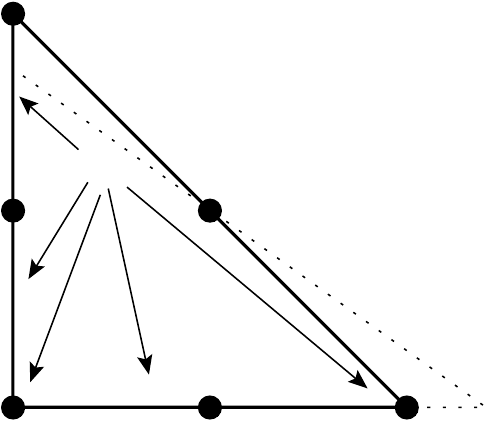
\else
\input{figures/figureType1toType2orTiltNew.pstex_t}
\fi}
\caption{In the proof of Lemma~\ref{lem:Type1}, Step~1, a Type 1 triangle
  can be replaced by a Type 2 triangle (dotted) that gives the same inequality.%
}\label{figure:simple2}
\end{figure}
\textbf{Step 2.} From Step 1, if $\sum_{j=1}^k\psi_B(r^j)s_j\geq 1$ is extreme, then $\#(\verts(B) \cap P) \geq 2$, i.e., there exist $p^1, p^2 \in \verts(B) \cap P$. As in the statement of this lemma, $p^1, p^2 \in F_3$.  If $P \cup \{f\} \subset S_3$, then $\sum_{j=1}^k\psi_B(r^j)s_j\geq 1$ is dominated or realized by the valid inequality derived from $S_3$. Therefore either Case a or Case b occurs.   
\end{proof}

\subsection{Type 2 triangles and splits} For these two types of maximal
lattice-free sets, we allow tilts where $\Y = (Y_1, \ldots, Y_n)$ may not be a
covering of $Y(B)$. This may create non-lattice-free sets in $\T(B, \Y) \cap
\S(B)$ as the hypothesis of Observation~\ref{obs:latticefree} is not
satisfied. We handle this by adding an additional edge to take care of the
conflicting lattice points in the interior. Recall the notation $v(F_i)$ for
the lattice vector which generates the sub-lattice of $\Z^2$ parallel
to~$F_i$.  Moreover, we recall that $(x^1, x^2)$ denotes the open line segment 
between $x^1$ and $x^2$. 

\begin{figure}%
\centering

\scalebox{0.85}{%
\ifpdf
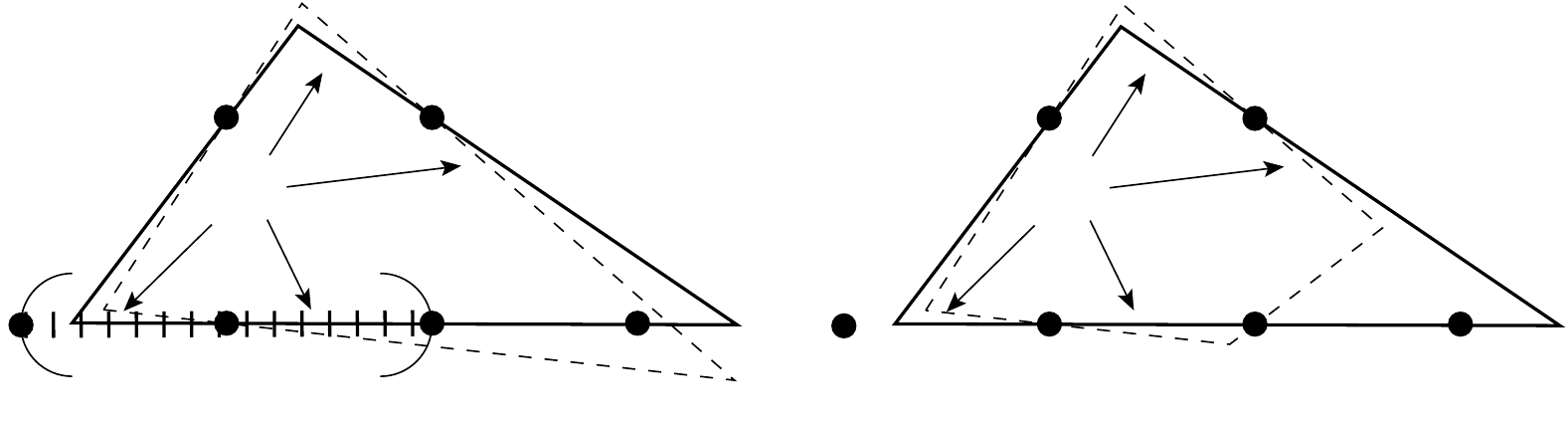
\else
\input{figures/figureType2non-latticeFreeProof.pstex_t}
\fi}
\caption{The geometry of Lemma~\ref{lemma:add_edge}.  (a)~The hypothesis of
the lemma regarding the ray intersections on~$F_3$.  (b)~A new edge is
constructed such that no rays point to it, turning the triangle to a quadrilateral.}
\label{fig:add_edge}
\end{figure}
\begin{lemma}
\label{lemma:add_edge}
Let $\M(B)$ be a Type 2 triangle with $\#(\conv(P \cap F_3) \cap \Z^2) \leq
1$. Suppose there exists a point $y^3 \in F_3\cap\Z^2$ such that $P \cap F_3
\subset (y^3 - v(F_3), y^3 + v(F_3))$.  Let $Y_i = \{y^i\}$, and suppose that
$\dim \T(B, 
\Y) \geq 1$. 

For any $\bar A \in \mathcal{N}(B, \Y)\setminus\{0\}$, there exists an $0 < \epsilon_1 < 1$ such that $\sum_{j=1}^k \psi_{B + \epsilon \bar A}(r^j)s_j\geq 1$ is a valid inequality for $\conv(\Rf)$ for every $0 < \epsilon \leq \epsilon_1$.
\end{lemma}
The geometry of this lemma is illustrated in Figure~\ref{fig:add_edge}\,(a).

\begin{proof}
Recall that a lattice-free set containing $f$ in its interior yields a valid
inequality for $\conv(\Rf)$. We will construct $0 < \epsilon_1 < 1$ such that for every $0 < \epsilon \leq \epsilon_1$ there exists a matrix $C = (c^1; c^2; c^3)$ with three rows or a matrix $C = (c^1; c^2; c^3;c^4)$ with four rows, such that $\M(C)$ is a lattice-free set and $\psi_C(r^j) = \psi_{B + \epsilon \bar A}(r^j)$ for $j=1, \dots, k$. Of course, in the case when $C$ has four rows, the set $M(C)$ will contain an additional edge. 

By Observation~\ref{obs:dimension}, there exists $0 <\delta < 1$ such that $B+\epsilon \bar A \in \T(B,\Y)\cap \S(B)$ for all $0 < \epsilon \leq \delta$. 
From the definition of $\S(B)$ it follows that $M(B+\epsilon\bar A)\cap \Z^2\subseteq Y(B)$ for all $0 < \epsilon \leq \delta$. Since $Y_1 = \{y^1\}$ and $Y_2 =\{y^2\}$, $y^1$ and $y^2$ are not contained in $\intr(M(B+\epsilon\bar A))$. This implies that  $\intr(M(B+\epsilon\bar A))\cap \Z^2\subset F_3$. 

If $\intr(M(B+\epsilon\bar A))\cap \Z^2 = \emptyset$ for every $0 < \epsilon \leq \delta$, then
$M(B+\epsilon\bar A)$ is lattice-free for every such $\epsilon$. So we let $\epsilon_1 = \delta$ and let $C = B + \epsilon \bar
A$ for every $0 < \epsilon \leq \delta$ and we are done. 

Otherwise, let $0 < \epsilon' \leq \delta$ be such that $\intr(M(B+\epsilon'\bar A))\cap \Z^2 \neq\emptyset$. Let $y^4$ be the closest integer point on $F_3$
to~$y^3$ such that $y^4\in \intr(\M(B+\epsilon'\bar A))$. Note that one can
then assume $y^4=y^3 + v(F_3)$. Next, pick $c^4 \in \R^2$ such that  $c^4
\cdot (x - f) \leq 1$ is a halfspace containing $P \cup \{y^1, y^2, y^3\}$ and
such that $c^4\cdot (y^4 - f) = 1$.  This exists because there are only
finitely many ray intersections, $y^4$ is on the boundary, and $P \cap F_3
\subset \{\,y^4 + t (y^3 - y^4)\st t > 0\,\}$ since $P \cap F_3 \subset (y^3 -
v(F_3), y^3 + v(F_3))$.\smallbreak 

Consider the set $$\mathcal V := \{\,(a^1; a^2; a^3) \in\R^{3\times 2}\st a^i \cdot r^j > c^4
\cdot r^j \text{ for}\ j = 1, \dots, k,\  i \in I_B(r^j)\,\}.$$   
Since $\mathcal V$ is an open set containing $B$, there exists $0 < \epsilon_1 \leq \epsilon'$ such that $B + \epsilon \bar A  \in \mathcal V$ for every $0 < \epsilon \leq \epsilon_1$. For any $0 < \epsilon \leq \epsilon_1$, let $(c^1; c^2; c^3) = B + \epsilon \bar A$. Then $C = (c^1; c^2; c^3; c^4)$ has the property that $\M(C)$ is a lattice-free
quadrilateral. This is because $\epsilon \leq \delta$ implies
$\intr(M(B+\epsilon \bar A))\cap \Z^2\subset F_3$. But all these integer points
violate the inequality $c^4 \cdot (x - f) \leq 1$.  See Figure~\ref{fig:add_edge}\,(b).

Moreover, $\psi_C(r^j) = \psi_{B + \epsilon \bar A}(r^j)$ for $j=1, \dots, k$. This is because $I_C(r^j)  = I_B(r^j) = I_{B+\epsilon \bar A}(r^j)$ for all $j$; the first equality follows because $B + \epsilon \bar A \in \mathcal{V}$ and the second equality follows from the fact that $B+\epsilon\bar A\in \T(B,\Y)$, since $\epsilon \leq \delta$. 
\end{proof}

One can prove an analogous lemma for splits. Although the statement and the proof are very similar to Lemma~\ref{lemma:add_edge}, there are some subtle differences. For example, $\S(B)$ is not full-dimensional when $\M(B)$ is a split; Lemma~\ref{lemma: S(B) full-dimensional} applies only when $M(B)$ is bounded. Hence, more work needs to be done to create a lattice-free set in this case.

\begin{lemma}
\label{lemma:add_edge2}
Let $\M(B)$ be a split with $\#(\conv(P \cap F_1) \cap \Z^2) \leq
1$. Let $y^1 \in F_1\cap\Z^2$ such that $P \cap F_1 \subset (y^1 - v(F_1), y^1 + v(F_1))$.  Let $Y_1 = \{y^1\}$ and $Y_2 = \{y^2, y^3\}$, where $y^2, y^3$ are two arbitrary integer points on $F_2$. Suppose that $\dim \T(B,\Y) \geq 1$. 

For any $\bar A \in \mathcal{N}(B, \Y)\setminus\{0\}$, there exists $0 < \epsilon_1 < 1$ such that $\sum_{j=1}^k \psi_{B + \epsilon \bar A}(r^j)s_j\geq 1$ is a valid inequality for $\conv(\Rf)$ for every $0 < \epsilon \leq \epsilon_1$.
\end{lemma}

\begin{proof}
Similar to the proof of Lemma~\ref{lemma:add_edge}, we will construct $0 < \epsilon_1 < 1$ such that for every $0 < \epsilon < \epsilon_1$, there exists a matrix $C = (c^1; c^2; c^3)$ such that $\M(C)$ is a lattice-free set containing one additional edge (so $\M(C)$ is a triangle) and $\psi_C(r^j) = \psi_{B + \epsilon \bar A}(r^j)$ for $j=1, \dots, k$. 

First, since $B$ satisfies the strict inequalities in $\T(B,\Y)$, there exists $0 < \delta < 1$ such that $B + \epsilon\bar A \in \T(B,\Y)$ for every $0 <\epsilon \leq \delta$.

Observe that setting $Y_2=\{\y^2, y^3\}$ implies that $F_2$ is fixed as the equalities in $\T(B,\Y)$ corresponding to $y^2, y^3$ force $F_2$ to lie on the line passing through $y^2, y^3$. Therefore, for any $\bar A \in \mathcal{N}(B,\Y)\setminus\{0\}$, $F_1$ is tilted for $M(B+\bar A)$ and hence $\M(B+\bar A)$ will contain lattice points in its interior. Let $y^4$ be the closest integer point on $F_1$ to $y^1$ such that $y^4\in \intr(\M(B+\bar A))$.  Note that one can then assume $y^4=y^1 + v(F_1)$. Choose $\hat y^2, \hat y^3 \in M(B + \bar A)\cap F_2$ such that $\hat y^2 - y^1$ and $v(F_1)$ form a lattice basis for $\Z^2$ and $\hat y^3 = \hat y^2 + v(F_1)$. This can be done because the equality conditions in $\T(B,\Y)$ from $Y_2$ fix the side $F_2$ of $M(B)$ and so it remains parallel to $v(F_1)$. Next, pick $c^3 \in \R^2$ such that  $c^3 \cdot (x - f) \leq 1$ is a halfspace containing $P \cup \{y^1, \hat y^2, \hat y^3\}$ and such that $c^3\cdot (y^4-f) = 1$.  This exists because there are only finitely many ray intersections, $y^4$ is on the boundary, and $P \cap F^1 \subset \{\,y^4 + t (y^1 - y^4)\st t > 0\,\}$  since $P\cap F_1 \subset (y^1 - v(F_1), y^1 + v(F_1))$.\smallbreak

Consider the set $$\mathcal V := \{\,(a^1; a^2) \in\R^{2\times 2}\st a^i \cdot r^j > c^3
\cdot r^j \text{ for}\ j = 1, \dots, k,\  i \in I_B(r^j)\,\}.$$   
Since $\mathcal V$ is an open set containing $B$, there exists an $0 < \epsilon_1 \leq \delta$ such that $B + \epsilon \bar A  \in \mathcal V$ for every $0 < \epsilon \leq \epsilon_1$. For any such $\epsilon$, let $(c^1; c^2) = B + \epsilon \bar A$.

We show that $C = (c^1; c^2; c^3)$ has the property that $\M(C)$ is a lattice-free triangle. Let $S$ be the split defined by the line passing through $y^1, \hat y^2$ and the line passing through $y^4, \hat y^3$ (this defines a split because $\hat y^2, \hat y^3, y^1$ and $y^4$ form a parallelogram of area 1). Since $\M(C)\cap \M(B) \subseteq \M(B)$, $\M(C)\cap \M(B)$ is lattice-free. Also, $\M(C)\setminus\intr(\M(B)) \subseteq S$ and hence $\M(C)\setminus\M(B)$ is lattice-free. Moreover the boundary shared by these two sets $\M(C)\cap \M(B)$ and $\M(C)\setminus \intr(\M(B))$ is the line segment $[y^1, y^4]$, which contains no integer points in its relative interior. Therefore, $\M(C)$ is lattice-free.

Moreover, $\psi_C(r^j) = \psi_{B + \epsilon \bar A}(r^j)$ for $j=1, \dots, k$ because $I_C(r^j) = I_B(r^j) = I_{B+\epsilon \bar A}(r^j)$ for all $j$. The first equality follows because $B+\epsilon \bar A \in \mathcal{V}$ and the second equality is because $\epsilon \leq \delta$ and so $B + \epsilon\bar A \in \T(B, \Y)$.
\end{proof}

With the above lemma, the necessary conditions for splits are easy to show.
\begin{lemma}[Splits]\label{lemma:split-conditions}
Suppose $\M(B)$ is a split.  If $\sum_{j=1}^k\psi_B(r^j)s_j\geq 1$ is extreme, then one of the following holds: \\ 
\textbf{Case a.} $P \subset \Z^2$.\\
\textbf{Case b.} There exists $j \in \{1, \ldots, k\}$ such that $r^j$ lies in the recession cone of the split.\\
\textbf{Case c.} $\#(\conv(P\cap F_i) \cap \Z^2) \geq 2$ for at least one of $i=1$ or $i=2$.
\end{lemma}

\begin{proof}
We suppose that we are not in Case a, Case b, or Case c and show that $\sum_{j=1}^k\psi_B(r^j)s_j\geq 1$ is not extreme. So we suppose, possibly by exchanging the labels on $F_1$ and $F_2$, that $F_1 \cap P \setminus \Z^2 \neq \emptyset$, no ray in $\{r^1, \ldots, r^k\}$ lies in the recession cone of the split, and
$\#(\conv(P\cap F_1) \cap \Z^2) \leq 1$.  

Let $y^1 \in F_1$ such that $P \cap F_1 \subset (y^1 - v(F_1), y^1 + v(F_1))$. Choose any $y^2, y^3 \in F_2 \cap
\Z^2$.  Let $Y_1 = \{ y^1\}, Y_2 = \{y^2, y^3\}$. Note that since we assumed that no ray lies in the recession cone, we have $|I_B(r^j)| = 1$, for every $j = 1, \ldots, k$. Hence, there are no equalities in $\T(B, \Y)$ for $I_B(r^j)$. Then $\dim \T(B, \Y) \geq 4
- 3 = 1$. Pick any $\bar A\in \mathcal{N}(B,\Y)\setminus\{0\}$. 

Notice that the equalities defining $\T(B,\Y)$ corresponding to
$y^2$ and $y^3$ fix $F_2$ completely because they force it to be the line
going through $y^2$ and $y^3$. In other words, $\bar a^2 = 0$. Therefore $\bar a^1 \neq 0$.

Since $B$ satisfies the strict inequalities of $\T(B,\Y)$, there exists $\delta > 0$ such that $B \pm \epsilon \bar A \in \T(B,\Y)$ for all $0 <\epsilon \leq \delta$, implying (amongst other things) that $I_{B\pm\epsilon\bar A}(r^j) = I_B(r^j)$ for all $j= 1, \ldots, k$. Using Lemma~\ref{lemma:add_edge2} with $\bar A$, we know that there exists
an $0 < \epsilon_1 < 1$ such that $\sum_{j=1}^k\psi_{B+\epsilon \bar A}(r^j)s_j \geq 1$ is a valid inequality for every $0 < \epsilon \leq \epsilon_1$. Similarly, using Lemma~\ref{lemma:add_edge2} with $-\bar A$, there exists
an $0 < \epsilon_2 < 1$ such that $\sum_{j=1}^k\psi_{B-\epsilon \bar A}(r^j)s_j \geq 1$ is a valid inequality for every $0 < \epsilon \leq \epsilon_2$. Let $\epsilon = \min\{\delta, \epsilon_1, \epsilon_2\}$. Thus, $\sum_{j=1}^k\psi_{B\pm \epsilon \bar A}(r^j) \geq 1$ are both valid inequalities. 

Since $\bar A \in \mathcal{N}(B,\Y)$, $\bar a^1\cdot (y^1-f) = 0$. Since $F_1 \cap P \setminus\Z^2 \neq \emptyset$, there exists $r^j$ with $I_B(r^j) = \{1\}$ and $p^j \not\in \Z^2$ and so $r^j$ and $y^1-f$ are linearly independent. This implies that $\bar a^1\cdot r^j \neq 0$ since $\bar a^1\cdot y^1 = 0$ and $\bar a^1 \neq 0$. Hence, $\psi_{B+\epsilon \bar A}(r^j) = (b^1 + \epsilon\bar a^1)\cdot r^j \neq (b^1 - \epsilon\bar a^1)\cdot r^j = \psi_{B-\epsilon \bar A}(r^j)$. The equalities follow because $\epsilon \leq \delta$ and so $I_{B\pm\epsilon\bar A}(r^j) = I_B(r^j) = \{1\}$. Moreover, since $B \pm \epsilon \bar A \in \T(B,\Y)$, Lemma~\ref{lemma:convex_combination} implies that $\sum_{j=1}^k\psi_B(r^j) \geq 1$ is a convex combination of the two valid inequalities $\sum_{j=1}^k\psi_{B\pm \epsilon \bar A}(r^j) \geq 1$. Hence, we have shown that $\sum_{j=1}^k\psi_B(r^j)s_j\geq 1$ is not extreme by using two Type 2 triangles (note that the triangle $M(C)$ in the proof of Lemma~\ref{lemma:add_edge2} is a Type 2 triangle).  
\end{proof}

\begin{figure}
\centering

\scalebox{0.85}{%
\ifpdf
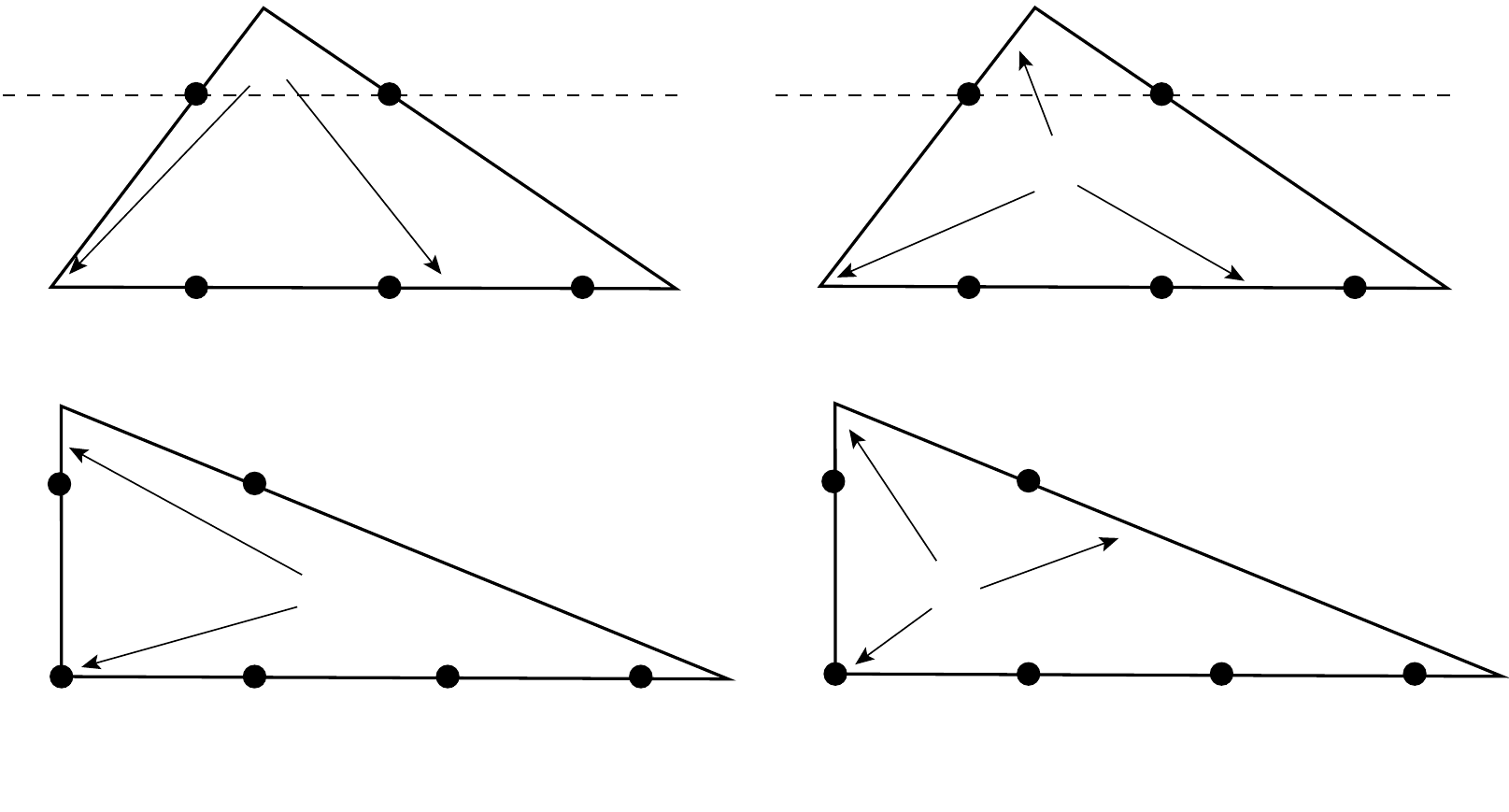
\else
\input{figures/figureType2Cases.pstex_t}
\fi}
\caption{Cases of extreme Type 2 triangles in Lemma~\ref{lemma:type 2 conditions}.}
\label{fig:Type2Cases}
\end{figure}

\begin{lemma}[Type 2 Triangles]\label{lemma:type 2 conditions}
Let $\M(B)$ be a Type 2 triangle with facets $F_1, F_2, F_3$ where $F_3$ is the facet containing multiple integer points.  Let $y^1, y^2$ be the unique integer points on the relative interiors of $F_1$ and $F_2$, respectively.  \\
 If $\sum_{j=1}^k\psi_B(r^j)s_j\geq 1$ is extreme and not dominated or realized by a split inequality, then one of the following holds:\\
\textbf{Case a.} $P \subset \Z^2$.\\
\textbf{Case b.} There exist $p^1, p^2 \in P \cap F_3$ with $\#([p^1,p^2] \cap
\Z^2) \geq 2$, and there exists a matrix $B'$ such that $M(B')$ is a Type 2
triangle, $\psi_{B'}(r^j) = \psi_B(r^j)$ for all $j = 1, \ldots, k$, and has
at least one of $p^1$ or $p^2$ in $\verts(B')$.  If there exist
non-integer-pointing rays on the relative interior of both $F_1, F_2$, then
there exist two corner rays.  Also, one of the following holds:\\ 
\indent \textbf{Case b\oldstylenums{1}.} $f \notin S_3$.\\
\indent \textbf{Case b\oldstylenums{2}.} $f \in S_3$ and $P \not\subset F_3$.\\
\textbf{Case c.}  There exist $p^1, p^2 \in P\cap F_i$ with $i=1$ or $i=2$, with $\#([p^1,p^2] \cap \Z^2) \geq 2$, such that $p^1 \in F_3 \cap \Z^2$ and if $P\setminus (F_i \cup F_3 \cup \Z^2) \neq \emptyset$, then $p^2$ can be taken to be a corner ray.
Also, one of the following holds:\\
\indent \textbf{Case c\oldstylenums{1}.} $f \notin S_i$.\\
\indent\textbf{Case c\oldstylenums{2}.} $f \in S_i$ and $P \not\subset S_i$.  
\end{lemma}
The cases of the lemma are illustrated in Figure~\ref{fig:Type2Cases}.

\begin{proof}

\noindent \textbf{Step 1.}
Suppose $P \not\subset \Z^2$ and there do not exist $p^1, p^2 \in P$ such that
$\#([p^1,p^2] \cap \Z^2) \geq 2$.  We will show that $\sum_{j=1}^k\psi_B(r^j)s_j\geq 1$ is then not
extreme.\smallbreak

First note that there is at most one corner ray in $F_3$ because there are
multiple integer points on~$F_3$.   Let $y^3 \in F_3$ such that $P \cap F_3
\subset (y^3 - v(F_3), y^3 + v(F_3))$.  Let $Y_i = \{y^i\}$.\smallbreak

\begin{figure}
\centering

\scalebox{0.85}{%
\ifpdf
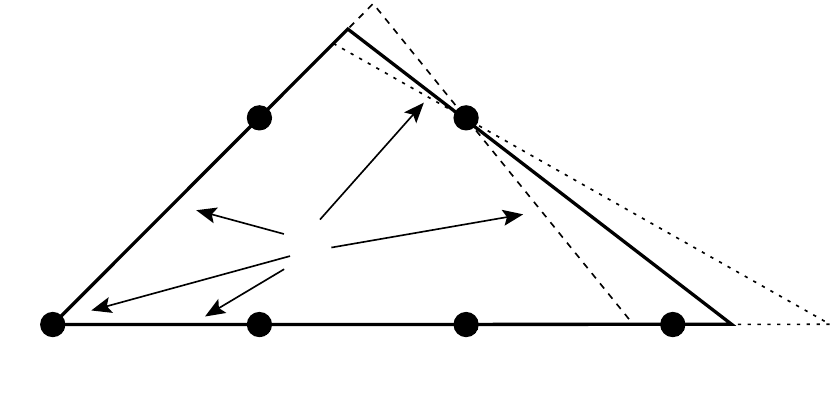
\else
\input{figures/figureType2Step1.pstex_t}
\fi}
\caption{In the proof of Lemma~\ref{lemma:type 2 conditions}, Step~1, a simple
  tilt from Lemma \ref{lemma:simple_tilts} shows that $\sum_{j=1}^k\psi_B(r^j)s_j\geq 1$ is not
  extreme.}
\label{fig:Type2Step1}
\end{figure}
Suppose first that $y^3 \in \verts(B) \cap P$ and, without loss of generality,
$y^3 \in F_1 \cap F_3$. Note that this implies that there are no corner rays on $F_2$, because $\#([p^1, p^2] \cap \Z^2) \leq 1$ and so $P \cap F_2 \subset \relint(F_2)$. If $P \cap F_2 \setminus \Z^2 \neq \emptyset$, then a
simple tilt from Lemma \ref{lemma:simple_tilts} shows that $\sum_{j=1}^k\psi_B(r^j)s_j\geq 1$ is not
extreme, as shown in Figure \ref{fig:Type2Step1}. If instead $P \cap F_2
\setminus \Z^2 = \emptyset$, then $P\subset \conv(\{y^1, y^2, y^3, y^4\})$, 
where $y^4$ is the integer point adjacent to~$y^3$ on~$F_3$,
since no two elements of $P$ contain two integer points between them.  Hence,
$P\cup \{f\} \subset S_i$ for either $i=1$ or $3$, and hence $\sum_{j=1}^k\psi_B(r^j)s_j\geq 1$ is
dominated by the inequality derived from $S_i$, contradicting the hypothesis of this
lemma.\smallbreak 

Suppose now that $y^3 \in \relint(F_3)$.  Since there are at most 2 corner
rays, Lemma~\ref{lemma:corner_rays_almost} shows that there exists $\bar A\in
\mathcal{N}(B,\Y)  \setminus\{0\}$ such that for every $0 < \epsilon < 1$, $\psi_{B+\epsilon \bar A}(r^j) \neq
\psi_{B-\epsilon \bar A}(r^j)$ for some $j = 1, \dots, k$ and $\psi_B(r^j) = \frac{1}{2} \psi_{B-\epsilon\bar A}(r^j) +
\frac{1}{2} \psi_{B+\epsilon \bar A}(r^j)$ for every $j = 1, \ldots, k$. If we pick $\epsilon$ arbitrarily, it is possible that
$M(B+\epsilon\bar A)$ or $M(B-\epsilon\bar A)$ is not lattice-free. However, using Lemma~\ref{lemma:add_edge} with $\bar A$ and $-\bar A$, we know that there exist $0 < \epsilon_1 < 1$ and $0 < \epsilon_2 < 1$ such that for $\epsilon = \min\{\epsilon_1, \epsilon_2\}$, both the inequalities
$\sum_{j=1}^k \psi_{B \pm \epsilon \bar A}(r^j)s_j\geq 1$ are valid for
$\conv(\Rf)$. Therefore $\sum_{j=1}^k\psi_B(r^j)s_j\geq 1$ is not extreme.

We comment here that, due to Lemma \ref{lemma:add_edge}, we may be using inequalities that come from quadrilaterals to show that $\sum_{j=1}^k\psi_B(r^j)s_j\geq 1$ is not extreme.

\smallbreak

Therefore, if $\sum_{j=1}^k\psi_B(r^j)s_j\geq 1$ is extreme, we are either in Case a with $P\subset \Z^2$, or there exist  $p^1, p^2 \in P$ with $\#([p^1,p^2] \cap \Z^2) \geq 2$.  In the latter case, we now show that we must be in either Case b\oldstylenums{1}, b\oldstylenums{2}, c\oldstylenums{1}, or c\oldstylenums{2}.\medbreak

\noindent \textbf{Step 2.}\ Suppose $P \not\subset \Z^2$ and there exist  $p^1, p^2 \in P \cap F_3$ with $\#([p^1,p^2] \cap \Z^2) \geq 2$. Without loss of generality, we label $p^1,p^2$ such that $P \cap F_3 \subset [p^1,p^2]$.\smallbreak

\begin{figure}
\centering

\scalebox{0.85}{%
\ifpdf
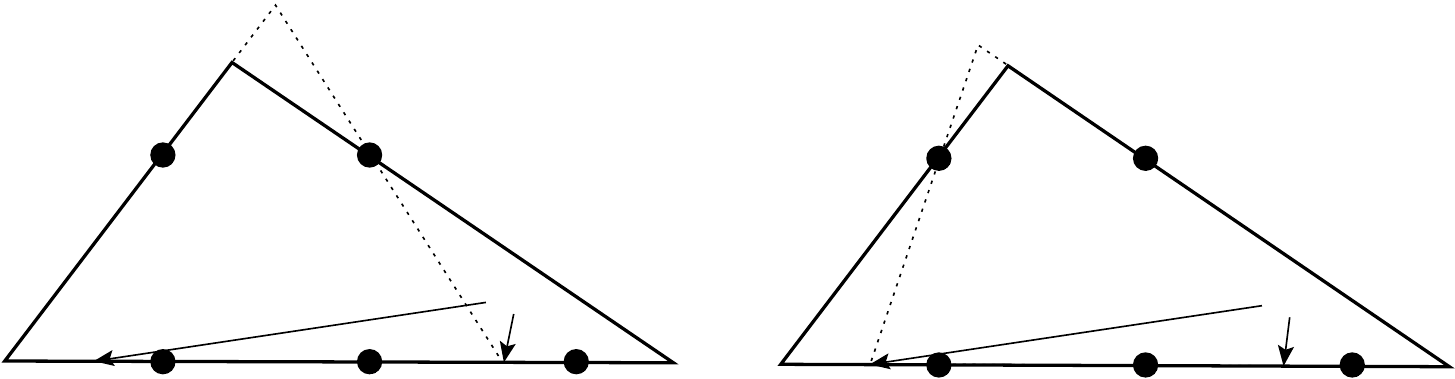
\else
\input{figures/figureType2Step2aNew.pstex_t}
\fi}
\caption{In the proof of Lemma~\ref{lemma:type 2 conditions}, Step~2a, either
  $F_1$ or $F_2$ is tilted to give a new triangle~$M(B')$ (dotted).   
  (a)~Here $F_2$ cannot be used because tilting would remove $f$ from the
  interior.
  (b)~Instead, $F_1$ needs to be used.}
\label{fig:type2-step2a}
\end{figure}
\textbf{Step 2a.}\ We will show that there exists a matrix $B'$ such that $M(B')$ is a lattice-free Type 2 triangle that has at least one corner ray in $F_3$, and $\psi_{B'}(r^j) = \psi_B(r^j)$ for all $j = 1, \ldots, k$.\smallbreak

If either $p^1$ or $p^2$ is a vertex of $M(B)$, then we let $B' = B$ and move to Step 2b. We now deal with the case that $p^1 \not\in \verts(B)$ and $p^2 \not\in \verts(B)$.

Suppose there exists $\hat r \in \{r^1, \ldots, r^k\}$ such that $\hat p \in F_1 \cap F_2$, i.e., $\hat r$ is a corner ray on $F_1$ and $F_2$. We now make a tilting space argument to argue that $\sum_{j=1}^k\psi_B(r^j)s_j \geq 1$ is not extreme. We define $\Y = (Y_1, Y_2, Y_3)$ as $Y_1 = \{y^1\}$, $Y_2 = \{y^2\}$ and $Y_3 = F_3 \cap Y(B)$. Hence, $\Y$ is a covering of $Y(B)$. Since there is only one corner ray ($p^1 \not\in \verts(B)$ and $p^2 \not\in \verts(B)$), only one equation in $\mathcal{N}(B,\Y)$ comes from a corner ray condition. $Y_1$ and $Y_2$ each contribute one equation. $Y_3$ contributes a system of equalities involving $a^3$ with rank $2$. Therefore, $\dim \mathcal{N}(B,\Y) = 6 - 5 = 1$. We pick any $\bar A \in \mathcal{N}(B,\Y)\setminus\{0\}$.  From Observation~\ref{obs:dimension} and Observation~\ref{obs:latticefree}, there exists $\epsilon > 0$ such that $\sum_{i=1}^k\psi_{B \pm \epsilon \bar A}(r^j)s_j \geq 1$ are both valid inequalities and Lemma~\ref{lemma:convex_combination} implies that $\sum_{j=1}^k\psi_B(r^j)s_j \geq 1$ is a convex combination of these two valid inequalities. We now show that $\psi_{B+ \epsilon \bar A}(\hat r) \neq \psi_{B-\epsilon \bar A}(\hat r)$. Note that the equations from $Y_3$ impose that $\bar a^3 = 0$. Therefore, either $\bar a^1 \neq 0$ or $\bar a^2 \neq 0$. Without loss of generality, assume $\bar a^1 \neq 0$. Observe now that $y^1 - f$ and $\hat r$ are linearly independent since $y^1$ is in the relative interior of $F_1$ and $\hat p$ is a vertex of $F_1$. Since $Y_1$ imposes $\bar a^1\cdot (y^1 - f)=0$, this implies that $\bar a^1\cdot \hat r \neq 0$. Therefore, $\psi_{B+ \epsilon \bar A}(\hat r) = (b^1 + \epsilon \bar a^1 )\cdot \hat r\neq (b^1 - \epsilon \bar a^1 )\cdot\hat r = \psi_{B-\epsilon \bar A}(\hat r)$; the equalities follow from the fact that $B \pm \epsilon \bar A \in \T(B,\Y)$ implying that $I_{B\pm\epsilon\bar A}(\hat r) = I_B(\hat r)$.

So we can assume that $p^1 \not\in \verts(B)$, $p^2 \not\in \verts(B)$ and $F_1 \cap F_2 \not\in P$, i.e., there is no corner ray in $M(B)$. Since $F_1$ and $F_2$ do not have corner rays, then we must have $\relint(F_i)  \cap P\setminus\Z^2 = \emptyset$ for $i=1,2$ because otherwise Lemma \ref{lemma:simple_tilts} shows that $\sum_{j=1}^k\psi_B(r^j)s_j\geq 1$ is not extreme, by a simple tilt of $F_1$ or $F_2$. 
For $i=1, 2$, since $\relint(F_i)  \cap (P \setminus\Z^2) = \emptyset$, changing $F_i$ to
now lie on the line through $p^i$ and $y^i$ does not change $\sum_{j=1}^k\psi_B(r^j)s_j\geq 1$, unless
$f$~is no longer in the interior of the set. At most one of these facet tilts puts $f$ outside the perturbed set, thus at
least one of them is possible.  This is illustrated in
Figure~\ref{fig:type2-step2a}. Without loss of generality, we assume that the
tilt of facet~$F_1$ is possible.  Let the set after tilting be $\M(B')$ and
$B'$ be the corresponding matrix. 

We claim that $\M(B')$ is lattice-free.  To
see this, let $y^3, y^4\in [p^1, p^2] \cap \Z^2$ be distinct integer points
adjacent to each other.  Then consider the split $S$ with facets through $[y^3, y^1]$ and
$[y^4, y^2]$.  Since $[y^3, y^4] \subset [p^1, F_1 \cap F_3]$ is a strict
subset, the new intersection at $F_1 \cap F_2$ is a subset of the split, and
hence $M(B') \setminus \M(B) \subset S$, and therefore no new integer points
are introduced. \smallbreak 

\textbf{Step 2b.}
Suppose now that $p^1 \in F_1 \cap F_3$ and there exists a point $p\in \relint(F_2) \setminus \Z^2$.  If there are no corner rays on~$F_2$, then Lemma \ref{lemma:simple_tilts} shows that $\sum_{j=1}^k\psi_B(r^j)s_j\geq 1$ is not extreme.  Therefore the conditions of Case b are met.  
If $P\cup \{f\}\subset S_3$  then $\sum_{j=1}^k\psi_B(r^j)s_j\geq 1$ is dominated or realized by the split inequality from $S_3$, hence either Case b\oldstylenums{1} or Case b\oldstylenums{2} occurs.\medbreak

\noindent\textbf{Step 3.}\ Suppose $P \not\subset \Z^2$ and there exist  $p^1,
p^2 \in P \cap F_i$ with $\#([p^1,p^2] \cap \Z^2) \geq 2$, for $i=1$ or $i=2$.
Without loss of generality, we assume that $i=1$.  In order for $\#([p^1,p^2]
\cap \Z^2) \geq 2$, it has to equal exactly two, and one of the points, say
$p^1$, must lie in $p^1 \in F_1 \cap F_3 \cap \Z^2$.  Thus, $p^1$ is the corner
ray.

\begin{figure}
\centering

\scalebox{0.85}{%
\ifpdf
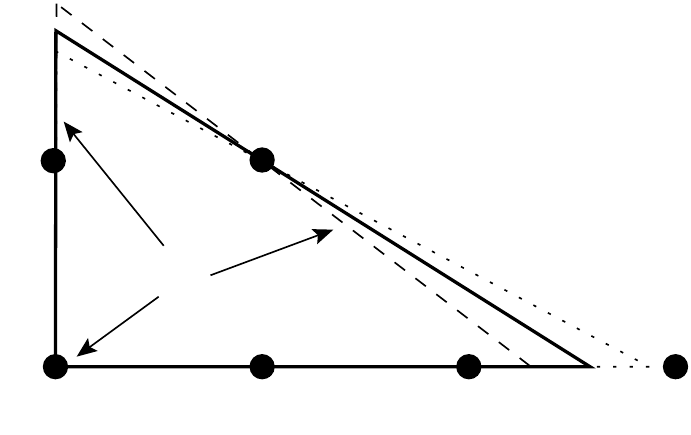
\else
\input{figures/figureType2Step3.pstex_t}
\fi}
\caption{In the proof of Lemma~\ref{lemma:type 2 conditions}, Step~3, a simple
  tilting argument (Lemma \ref{lemma:simple_tilts}) shows that the inequality
  is not extreme.}
\label{fig:Type2Step3}
\end{figure}
If there exists a point $p \in \relint(F_2) \setminus \Z^2$, then again, there
must be a corner ray on~$F_2$; otherwise, Lemma \ref{lemma:simple_tilts} shows
that $\sum_{j=1}^k\psi_B(r^j)s_j\geq 1$ is not extreme.  See Figure~\ref{fig:Type2Step3}. Since we are not in Case b, this must be the corner ray pointing to $F_1 \cap F_2$. Thus $p^2$ can be taken to be this corner ray.


As in Case b, if $P\cup \{f\} \subset S_1$, then $\sum_{j=1}^k\psi_B(r^j)s_j\geq 1$ is dominated or realized by the inequality derived from $S_1$. Hence, we are either in Case c\oldstylenums{1} or Case c\oldstylenums{2}. \smallskip

This concludes the proof.
\end{proof}

\section{Number of facets of the integer hull}\label{sec:poly_facets}

We recall that we have $k$ rays $r^1,\ldots ,r^k$.

\begin{remark}
\label{rem:ray_cone}
Given two rays $r^1$ and $r^2$ in $\mathbb{R}^2$, 
we denote by $C(r^1, r^2)$ the cone 
$\{\,x \in \mathbb{R}^2 \st x = f + s_1r^1 + s_2r^2, \mbox{with } 
s_1,s_2 \geq 0\,\}$. By Theorem \ref{THM:2ineqs}, we
get that $(C(r^1,r^2))_\mathrm{I}$ has a polynomial number of facets and
vertices.
\end{remark}

\begin{theorem}
\label{THM:POLYFACETS}
The number of facets of $\conv(\Rf)$ is polynomial in the size of the
encoding of the problem for $m=2$.
\end{theorem}

\begin{proof}

We will follow the cases from section \ref{sec:nec_cond} for each type of maximal lattice-free convex set in $\R^2$.\smallbreak

We will first handle the case where $P \subset \Z^2$.  That is, let $P$ be the
set of closest integer points that the rays point to from $f$.  If $\conv(P)$
is a lattice-free set, then it is contained within a maximal lattice-free set.
Choose any particular maximal lattice-free set containing $P$.  This covers
Case a for Type 2 and 3 triangles, quadrilaterals, and splits.  We will no
longer refer to this Case a for these types of lattice-free sets.\medbreak

\noindent\textbf{Splits.} The necessary conditions are given in
Lemma~\ref{lemma:split-conditions}. We consider the two remaining cases, which
are illustrated in Figure~\ref{fig:count-splits}.\smallskip  
\begin{figure}
\centering

\scalebox{0.85}{%
\ifpdf
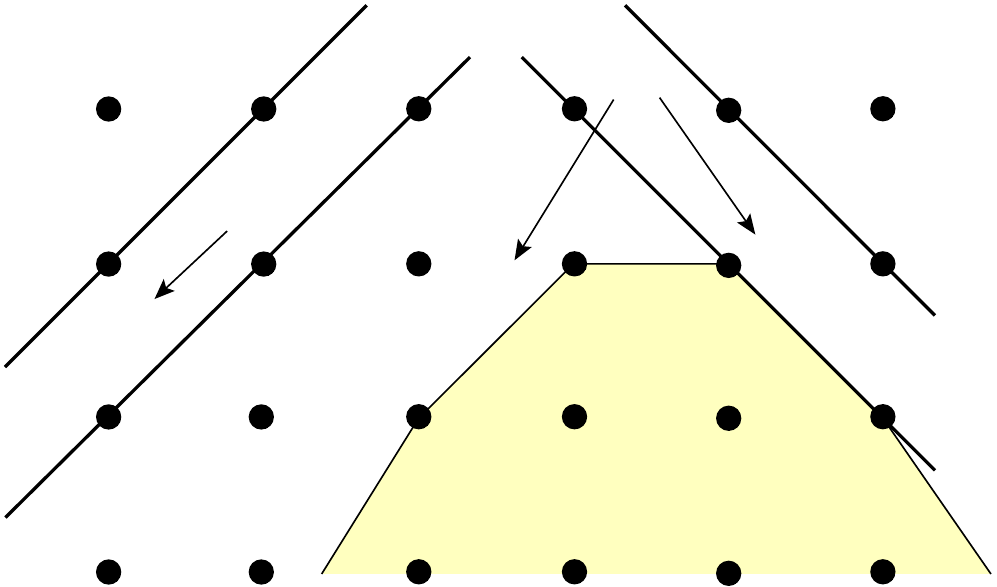
\else
\input{figures/figureSplitFacetParallel.pstex_t}
\fi}
\caption{Counting a polynomial number of splits}\label{fig:count-splits}
\end{figure}

\indent \textit{Case b.}  A ray
direction~$r^j$ is parallel to the split.  There are at most $k$ such
ray directions, and thus at most~$k$ splits in this case.\smallbreak

\indent \textit{Case c.} There exist $p^1, p^2$  such that $[p^1, p^2] \cap \Z^2 \geq 2$, and therefore, the split must run parallel to a facet of $(C(r^1,r^2))_\mathrm{I}$, of which there are only polynomially many. There are only ${ k \choose 2}$ ways to choose two rays for this possibility.\medbreak

\begin{figure}
\centering

\scalebox{0.85}{%
\ifpdf
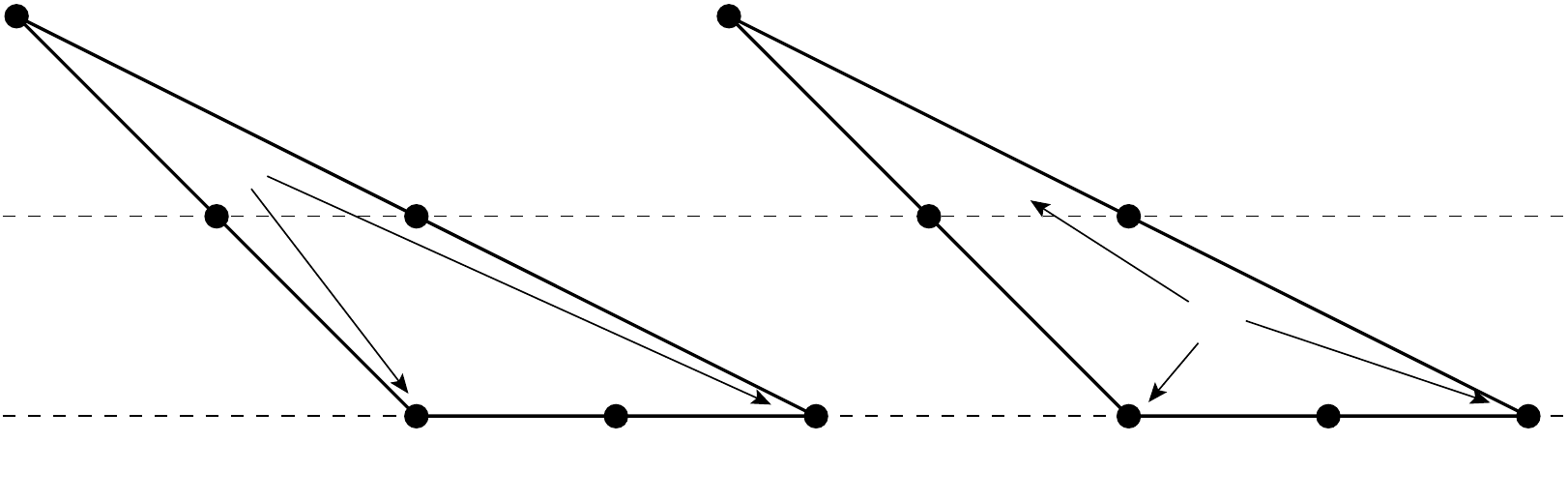
\else
\input{figures/figureType1Choices.pstex_t}
\fi}
\caption{Counting a polynomial number of Type 1 triangles}\label{fig:count-type1}
\end{figure}

\noindent\textbf{Type 1 triangles.} 
We assume that the inequality cannot be realized or dominated by a Type 2
triangle or split, because in this case we will use the analysis for these two
types. We now apply Lemma~\ref{lem:Type1} and refer to Figure~\ref{fig:count-type1}.

There are two corner rays, call them $r^1, r^2$; there are ${k \choose 2}$
ways to choose them.  Since these rays both point directly to integer points, they uniquely define $F_3$.\smallbreak

\indent\textit{Case a.} Since $f$ does not lie in the split~$S_3$, the integer
points $y^1, y^2$ are uniquely determined.\smallbreak

\indent\textit{Case b.} Since $f$ lies in the split~$S_3$ and there exists a
ray intersection~$p^3$ outside the split, the integer points $y^1, y^2$ are uniquely determined.\smallbreak

In both cases, since $F_3, y^1, y^2$, and the corner rays $r^1, r^2$ uniquely determine the
triangle, there are only polynomially many Type 1 triangles that we must
consider. 
\medbreak

\begin{figure}
\centering

\scalebox{0.85}{%
\ifpdf
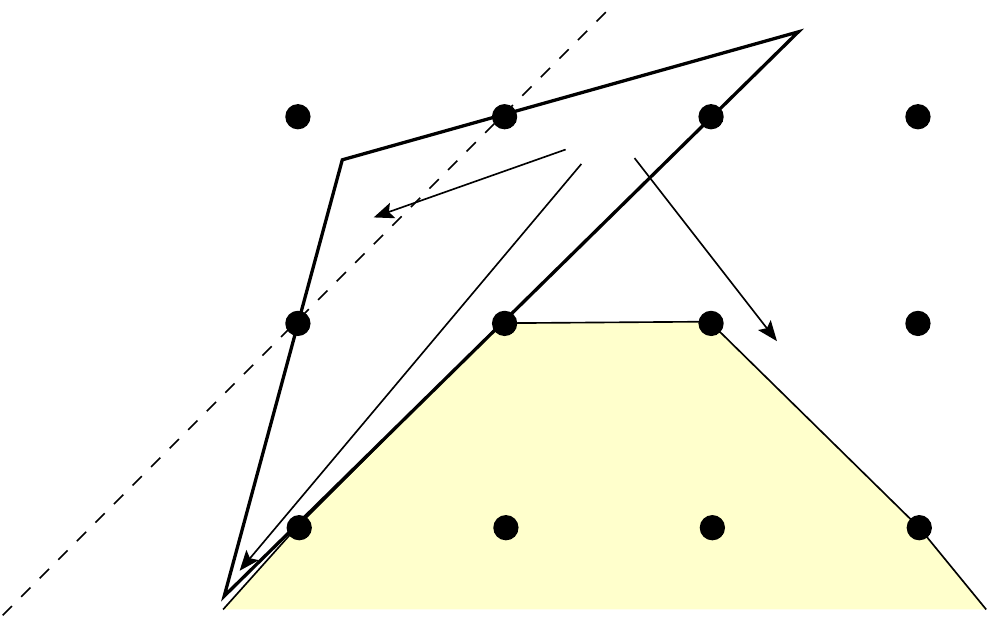
\else
\input{figures/figureType2ChoicesNew.pstex_t}
\fi}
\caption{Counting a polynomial number of Type 2 triangles in Case b}
\label{fig:Type2ChoicesNew}
\end{figure}

\noindent\textbf{Type 2 triangles.}  The necessary conditions are given in
Lemma~\ref{lemma:type 2 conditions}. \smallskip

\indent\textit{Case b}. 
We first pick the two rays $r^1, r^2$ to be the rays that are closest to
$F_1\cap F_3$ and $F_2 \cap F_3$, respectively. This can be done in $ {k
  \choose 2}$ ways.  See Figure~\ref{fig:Type2ChoicesNew}.

We next pick the facet $F_3$ as a facet of $(C(r^1,r^2))_\mathrm{I}$, which can be done only polynomially many ways.   

Now we choose $y^1, y^2$.  In Case b\oldstylenums{1}, where $f \not\in S_3$,
they are given uniquely by where $f$ is.  In Case b\oldstylenums{2}, when $P
\not\subset S_3$, we first pick a ray $r^3$ such that the corresponding ray
intersection $p^3$ will be the one that is not contained in $S_3$, and so
$r^3$ points between $y^1$~and~$y^2$. This would imply that $y^i$ is one of the vertices of $(C(r^i, r^3))_\mathrm{I}$. Moreover, since $y^1, y^2$ have to lie on the lattice plane adjacent to $F_3$, we have a unique choice for $y^1, y^2$ once we choose $r^3$. Now $r^3$ can be chosen in $O(k)$ ways and so there are $O(k)$ ways to pick $y^1, y^2$.

If we choose there to be a second corner ray somewhere (we can do this in $O(k)$ ways), then the triangle is uniquely determined by the two corner rays, $F_3$, $y^1$, and $y^2$.

On the other hand, if we choose that there is only one corner ray, then we pick $r^1$ or $r^2$ to be the only corner ray (2 choices), and the facet opposite of this corner ray cannot have any rays pointing to it that do not point to an integer point.  This is because that facet has no corner rays.  Therefore, any particular choice of this facet with no rays pointing to it will suffice (although one may not exist).

Hence, there are only polynomially many possibilities for Case b.\smallbreak

\begin{figure}
\centering

\scalebox{0.85}{%
\ifpdf
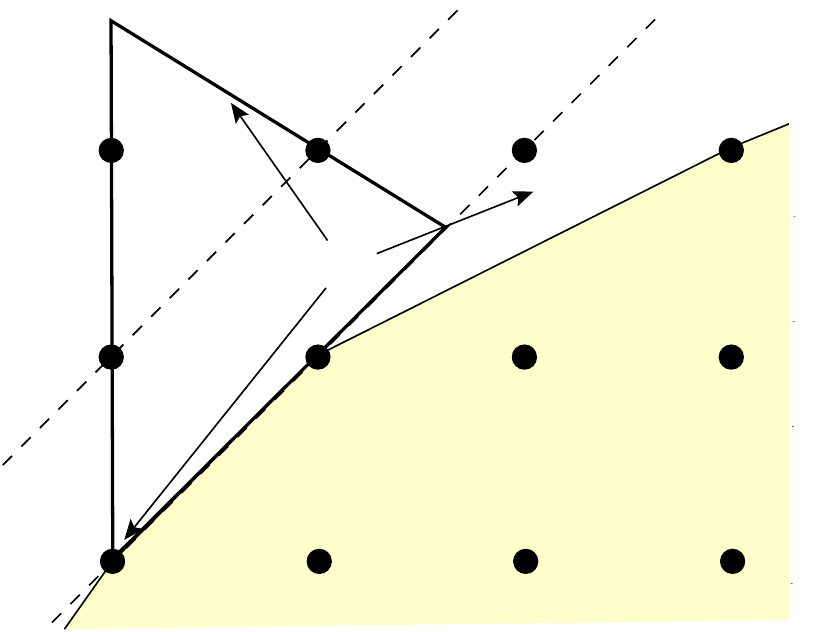
\else
\input{figures/figureType2ChoicesC.pstex_t}
\fi}
\caption{Counting a polynomial number of Type 2 triangles in Case c}
\label{fig:Type2ChoicesC}
\end{figure}
\indent \textit{Case c}.
We first choose $r^1, r^2$ to be the two rays such that $\#([p^1,
p^2]\cap\Z^2) \geq 2$.  One of them must point to an integer point on the
facet $F_3$. There are $2 \times {k \choose 2}$ ways to choose this.  Without
loss of generality, let $r^1$ point to the integer point on $F_3$.  See
Figure~\ref{fig:Type2ChoicesC}.

We next choose the facet $F_1$ from $(C(r^1,r^2))_\mathrm{I}$. There is a unique choice for $F_1$ because $p^3$ is an integer point and so $p^3$ will be {\em the} vertex of $(C(r^1,r^2))_\mathrm{I}$ (if one exists) that lies on the facet of $C(r^1, r^2)$ defined by the ray $r^1$. Hence $F_1$ can be the unique facet that is adjacent to this vertex but not lying on the facet of $C(r^1, r^2)$ defined by the ray $r^1$.

Now we pick $y^2, y^4$. This analysis is the same as with Cases b\oldstylenums{1} and b\oldstylenums{2}. In Case c\oldstylenums{1}, these points are uniquely determined by $f$.  In Case c\oldstylenums{2}, these are uniquely determined by one of the rays pointing between them. Thus, $y^2, y^4$ can be chosen in $O(k)$ ways after choosing this ray.

If we assume there are two corner rays ($r^1$ and $r^2$), then the triangle is uniquely determined by these corner rays, $F_1$, $y^2$, and $y^4$.

On the other hand, if we assume that $r^1$ is the only corner ray, then there
cannot be any rays pointing to the interior of the opposite facet $F_2$.
Therefore, this facet can be chosen to be any particular facet (if one exists) that does not
have rays pointing to it. Then the triangle is uniquely determined by $r^1, F_1, F_2, y^2$, and $y^4$.

Therefore, there are only polynomially many Type 2 triangles of Case c, and hence there are only polynomially many Type 2 triangles that we need to consider.\medbreak

\noindent\textbf{Type 3 triangles.} The necessary conditions are given in
Lemma~\ref{lemma:type 3 conditions}.\smallskip

\indent\textit{Case b.} We only need to consider Case b, where there are three corner rays.  Now we pick any triplet of rays, say $r^1,r^2,r^3$, and require that each side of
$\M(B)$ passes through a vertex of $(C(r^i,r^{i+1}))_\mathrm{I}$, $i = 1,2,3$ and
$r^4 = r^1$. There are only polynomially such triplets of integer vertices
$y^1, y^2, y^3$ to choose.  

We note that a triangle whose 3 corner rays and a point on the
relative interior of each facet are known is already uniquely determined.  In
the appendix, we prove this claim (Proposition~\ref{prop:unique-triangle}). 
%
%
Thus, we can use a triplet of rays and a vertex from
each integral hull of the three cones spanned by consecutive rays to
define the triangle. These are polynomial in number.\medbreak

\noindent\textbf{Quadrilaterals.} The necessary conditions are given in
Lemma~\ref{lemma:quadrilateral cond}.\smallskip

\indent\textit{Case b.}  
We first pick four rays $r^1, r^2, r^3, r^4$ to be corner rays, which can be done in ${ k \choose 4 }$ ways.  
We next pick four integer points $y^1, y^2, y^3, y^4$, with $y^i$ a vertex of $(C(r^i, r^{i+1}))_{\mathrm{I}}$, with $i=1,2,3$ and $y^4$ a vertex of $(C(r^4, r^1))_{\mathrm{I}}$.  This can be done in polynomially many ways.


Lemma \ref{lemma:quadrilateral cond} Case b says that if $\sum_{j=1}^k\psi_B(r^j)s_j\geq 1$ is extreme, then it is the unique quadrilateral with these corner rays and integer points.  Therefore, we count at most one quadrilateral for each set of corner rays and integer points.  
\smallbreak

Therefore, there are only polynomially many quadrilaterals that must be
considered.\medbreak

We have enumerated all the types of maximal lattice-free convex sets in $\R^2$ and shown that there are only polynomially many sets of each type that must be considered. Hence, for the case of $m=2$, we have shown that $\Rf$ has only polynomially many facets.
\end{proof}

We obtain the following result as a direct consequence of our proof for Theorem~\ref{THM:POLYFACETS}.
\begin{theorem}\label{thm:enumerate}
There exists a polynomial time algorithm to enumerate all the facets of $\conv(\Rf)$ when $m=2$.
\end{theorem}

\begin{proof}
For each of the five types of maximal lattice-free sets in the plane, the proof for Theorem~\ref{THM:POLYFACETS} shows how to generate in polynomial time the ones that are potentially facet defining. However, since we only ensure that the necessary conditions from Section~\ref{sec:nec_cond} are not violated, we can potentially generate a set of valid inequalities (of polynomial size) which is a superset of all the facets. We can then use standard LP techniques to select the facet defining ones from these.
\end{proof}

\clearpage
\appendix
{\small
\section{Appendix: Uniqueness of a triangle defined by 3 corner rays and a point on the
  relative interior of each facet}

%
%
%
%

%

%


\begin{proposition}\label{prop:unique-triangle}
Any triangle defined by 3 corner rays and 3 points (one on the relative interior of each facet) is uniquely defined. 
\end{proposition}
\begin{proof}
The space of these three corner rays and 3 points is exactly the tilting space
of any such triangle satisfying this.  For convenience we define $\y^i:= y^i - f$
and $\p^i := p^i - f$, where $p^i
$ are the ray intersections.  Then $\p^i = 
\frac{1}{\psi_B(r^i)} r^i$. 

We want to show that the solution to the following systems of equations is unique.
\begin{equation*}
	\begin{array}{ccc}
		\begin{array}{rl}
		 	a^1\cdot \y^1  &= 1            \\			
			 a^1\cdot \p^2  &=a^2\cdot  \p^2   \\
			a^2\cdot  \y^2  &= 1            \\
			a^2\cdot  \p^3  &= a^3\cdot  \p^3  \\
			a^3\cdot  \y^3  &= 1            \\
			a^3\cdot  \p^1  &=a^1\cdot  \p^1 
		\end{array}
	& \Rightarrow &
	\begin{bmatrix}
		\y^1                         \\
		 \p^2 & -\p^2              \\
			 & \y^2               \\
			 & \p^3 & - \p^3    \\
			 &       & \y^3      \\
		-\p^1 &       & \p^1
	\end{bmatrix}
	\begin{bmatrix}
		 a^1 \\ a^2 \\ a^3
	\end{bmatrix}
	= 
	\begin{bmatrix}
		1\\0\\1\\0\\1\\0
	\end{bmatrix}
\end{array}
\end{equation*}
We then write this down as a matrix equation where every vector in the matrix is a row vector of size 2, therefore we have a $6 \times 6$ matrix.  We will analyze the determinant of the matrix.

Since the points $\y^1, \y^2, \y^3$ are on the interior of each facet, they can be written as convex combinations of $\p^1, \p^2, \p^3$.

\begin{equation*}
\begin{array}{ccc}
\y^1 = \frac{1}{\alpha'} \p^1 + \frac{\alpha}{\alpha'} \p^2 &  & \p^1 = \alpha' \y^1 - \alpha \p^2\\
\y^2  =  \frac{1}{\beta'} \p^2 + \frac{\beta}{\beta'}\p^3&  \Rightarrow & \p^2 = \beta' \y^2 - \beta \p^3\\
\y^3 =  \frac{1}{\gamma'} \p^3 +\frac{\gamma}{\gamma'} \p^1&  & \p^3 = \gamma' \y^3 - \gamma \p^1
\end{array}
\end{equation*}

Therefore, we can perform row reduction on the last row.  Just tracking the last row, we have
$$
\begin{bmatrix}
-\p^1 & 0 & \p^1
\end{bmatrix}\rightarrow
\begin{bmatrix}
0 & \alpha \p^2 & \p^1
\end{bmatrix}
\rightarrow
\begin{bmatrix}
0& 0& \p^1 - \alpha \beta \p^3
\end{bmatrix}.
$$

This matrix now has an upper block triangular form, and the determinant is easily computed as 
$$
\det (\y^1; \p^2) \det(\y^2; \p^3) \det(\y^3; \p^1 - \alpha \beta \p^3).
$$
The first two determinants are non-zero because those vectors are linearly independent.  The last determinant requires some work:
$$
\begin{bmatrix}
\y^3\\
 \p^1 + \alpha \beta \p^3
\end{bmatrix} =
 \begin{bmatrix}
\frac{1}{\gamma'} \p^3 +\frac{\gamma}{\gamma'} \p^1\\
 \p^1 - \alpha \beta \p^3
\end{bmatrix}
= 
\begin{bmatrix}
\frac{\gamma}{\gamma'} & \frac{1}{\gamma'} \\
1 & -\alpha \beta
\end{bmatrix}
\begin{bmatrix}
\p^1\\
\p^3
\end{bmatrix}.
$$
Since all the coefficients are positive, the determinant of the first matrix is strictly negative, and since $\p^1, \p^3$ are linearly independent, the determinant of the second matrix is non-zero.\smallbreak

Hence, the determinant of the original matrix is non-zero, and therefore the system of equations has a unique solution.
\end{proof}
}

\clearpage
\paragraph*{Acknowledgments.}
A.~Basu wishes to thank G{\'e}rard Cornu{\'e}jols and Fran\c{c}ois Margot
for many discussions which led to an earlier
manuscript~\cite{bcm:strongest-2row-unpublished}, which the present article is
based on.
During the completion of this work, R.~Hildebrand was supported by grant
DMS-0636297 (VIGRE), and R.~Hildebrand and M.~K\"oppe were supported by grant
DMS-0914873 of the National Science Foundation.

\bibliography{../bib/MLFCB_bib}{}
\bibliographystyle{../amsabbrv}

\end{document}